\documentclass[11pt]{article}

\usepackage[T1]{fontenc}
\usepackage{amsfonts}
\usepackage{amsmath}
\usepackage{amssymb}
\usepackage{amsthm}
\usepackage{bbm}
\usepackage{bm}
\usepackage{mathrsfs}
\usepackage{verbatim}
\usepackage{setspace}
\usepackage{color}
\usepackage{pdfsync}
\usepackage{enumitem}
\usepackage{graphicx}
\usepackage{stmaryrd}
\usepackage{float} %
\usepackage{tikz}
\usetikzlibrary{automata, positioning}
\usepackage{tcolorbox}

\usepackage[margin=1.25in]{geometry}

\theoremstyle{plain}
\newtheorem{theorem}{Theorem}[section]
\newtheorem{proposition}[theorem]{Proposition}
\newtheorem{lemma}[theorem]{Lemma}

\theoremstyle{definition}

\newtheorem{remark}[theorem]{Remark}

\theoremstyle{remark}

\renewenvironment{thebibliography}[1]{%
\begin{oldthebibliography}{#1}%
\setlength{\baselineskip}{.9em}
\linespread{1}
\small
\setlength{\parskip}{0.3ex}%
\setlength{\itemsep}{.2em}%
}%
{%
\end{oldthebibliography}%
}
\newcommand{\eps}{\varepsilon}

\newcommand{\R}{\mathbb{R}}

\newcommand{\cF}{\mathcal{F}}

\newcommand{\cN}{\mathcal{N}}

\DeclareMathOperator*{\argmin}{arg\, min}
\DeclareMathOperator*{\argmax}{arg\, max}

\newcommand{\tF}{\tilde{F}}
\newcommand{\tR}{\tilde{R}}

\newcommand{\mykill}[1]{}

\numberwithin{equation}{section}

\usepackage[pdfborder={0 0 0}]{hyperref}
\hypersetup{
  urlcolor = black,
  pdfauthor = {Marcel Nutz, Yuchong Zhang},
  pdfkeywords = {Mean Field Game; Stochastic Contest; Optimal Contract; Stackelberg Game},
  pdftitle = {Mean Field Contest with Singularity},
  pdfsubject = {Mean Field Contest with Singularity},
  pdfpagemode = UseNone
}

\begin{document}

\title{\vspace{-2.2em}
  Mean Field Contest with Singularity
\date{\today}
\author{Marcel Nutz\thanks{Departments of Statistics and Mathematics, Columbia University, New York, United States, \texttt{mnutz@columbia.edu}. Research supported by an Alfred P.\ Sloan Fellowship and NSF Grant DMS-1812661.}
\and
Yuchong Zhang\thanks{Department of Statistical Sciences, University of Toronto,
Canada, \texttt{yuchong.zhang@utoronto.ca}. Research supported by NSERC Discovery Grant RGPIN-2020-06290.}
}
}
\maketitle \vspace{-2.2em}

\begin{abstract}
We formulate a mean field game where each player stops a privately observed Brownian motion with absorption. Players are ranked according to their level of stopping and rewarded as a function of their relative rank. There is a unique mean field equilibrium and it is shown to be the limit of associated $n$-player games. Conversely, the mean field strategy induces $n$-player $\varepsilon$-Nash equilibria for any continuous reward function---but not for discontinuous ones. In a second part, we study the problem of a principal who can choose how to distribute a reward budget over the ranks and aims to maximize the performance of the median player. The optimal reward design (contract) is found in closed form, complementing the merely partial results available in the $n$-player case. We then analyze the quality of the mean field design when used as a proxy for the optimizer in the $n$-player game. Surprisingly, the quality deteriorates dramatically as $n$ grows. We explain this with an asymptotic singularity in the induced $n$-player equilibrium distributions.
\end{abstract}

\vspace{0.4em}

{\small

}

\section{Introduction}

We formulate a mean field game where each player stops a privately observed Brownian motion with drift and absorption at the origin. Players are ranked according to their level of stopping and paid a reward which is a decreasing function of the rank. This is an infinite-player version of the $n$-player game studied in~\cite{NutzZhang.21a} which in turn extends the Seel--Strack model~\cite{SeelStrack.13} where only the top-ranked player receives a reward.
First, we establish existence and uniqueness of a mean field equilibrium for any given reward function. Second, we solve the problem of optimal reward design (optimal contract) for a principal who can choose how to distribute a given reward budget over the ranks and aims to maximize the performance (i.e., stopping level) at a given rank, for instance the median performance among the players.
An analogous problem was studied for the $n$-player case in~\cite{NutzZhang.21a}, but only a partial characterization of the optimal design is available. Here, taking the mean field limit enables a clear-cut answer.

The present work also serves as a case study: from the perspective of mean field analysis, a particular feature of this game is to be tractable without necessarily being smooth. Atoms occur naturally in the equilibrium distribution, but the mean field game nevertheless admits an equilibrium that can be described in closed form, and we can prove analytically that the equilibrium is unique. As the $n$-player equilibrium can also be described in detail, we can observe the quality of the mean field approximation, not only for the mean field game (with fixed reward function) but also for the reward design problem---which is a Stackelberg game between the principal and a continuum of players. It turns out that this case study offers a cautionary tale.
 
In the $n$-player setting, the Seel--Strack model has been generalized and varied in several directions:  more general diffusion processes~\cite{FengHobson.15}, random initial laws~\cite{FengHobson.16a}, heterogeneous loss constraints~\cite{Seel.15}, behavioral players~\cite{FengHobson.16b}. See also~\cite{FangNoeStrack.20,NutzZhang.21a} for references to other models on risk-taking under relative performance pay, and~\cite{Vojnovic.2016} for an introduction to rank-order prize allocation. The novelty in the present work is to analyze a mean field model along the lines of Seel--Strack and its optimal design problem; we focus on the original Brownian dynamics. For the theory and applications of mean field games, the monographs \cite{BensoussanFrehseYam.13, CarmonaDelaRue.17a, CarmonaDelaRue.17b} provide an excellent overview and references. The very recent mean field model~\cite{AnkirchnerKaziTaniWendtZhou.21} can be related to the first part of this work. %
In~\cite{AnkirchnerKaziTaniWendtZhou.21}, players control the volatility of a Brownian motion up to an independent exponential time and are then ranked. The reward is one above a certain rank and zero below. As the horizon is exponential and volatilities can only be chosen within an interval that is bounded and bounded away from zero, the model has a smooth equilibrium and the distinct features of the present work do not appear. (Questions of optimal design, or reward functions other than the binary one, are not studied.) Contracts between a principal and $n$~agents have been analyzed in~\cite{GreenStokey.83}, among many others. Closer to the present work, \cite{ElieMastroliaPossamai.19} studies optimal contracts between a principal and infinitely agents, using the theory of mean field games for diffusion control.

\paragraph{Mean Field Equilibrium.}
The mean field game as formalized in Section~\ref{se:MFE} admits an equilibrium as soon as the reward function~$R$ is right-continuous and a natural integrability condition on the drift parameter holds (the latter is also present in the $n$-player game). The equilibrium stopping distribution can be described in closed form using the right-continuous inverse of $y\mapsto R(1-y)$; cf.~Theorem~\ref{thm:MFE}. Once the correct Ansatz is guessed, the existence result is reduced to a verification proof following a direct martingale argument. The uniqueness result is more involved, in part because---in contrast to the $n$-player game and many other mean field models---atoms in the equilibrium cannot be excluded; in fact, the closed-form solution already indicates that atoms will arise unless~$R$ is strictly monotone. The first part of the proof (Section~\ref{se:constantRewardsAndAtoms}) relates flat stretches in reward to atoms in any potential equilibrium distribution. The basic idea is to show, a priori, that equal pay must correspond to equal performance: ranks with the same reward are occupied by players that stop at the same level, and vice versa. On the other hand, jumps in reward are related to gaps in the support of the equilibrium distribution.
The second part of the uniqueness proof (Section~\ref{se:uniquenessByMinimization}) is based on the idea that in any equilibrium, the opposing players collectively act such as to minimize the value function of a given representative player. This approach enables an analytic proof using optimal stopping theory and dynamic programming arguments: using the additional constraints shown in the first part, the minimization is shown to have a unique solution, proving uniqueness of the equilibrium. This analysis is complicated by the presence of atoms. We remark that a similar uniqueness proof could be given for the $n$-player game, in which case it would simplify substantially because atoms can be excluded a priori (however, a different proof is already available).

\paragraph{Optimal Reward Design.}
In the $n$-player game, \cite{NutzZhang.21a} studied the design problem for a principal maximizing the performance at the $k$-th rank; for example, maximizing the revenue in a second-best auction, or the median performance among employees, customers, students, etc. To consider the analogue in the mean field limit, we replace the $k$-th rank by the quantile $\alpha=k/n$, for instance $\alpha=0.5$ for the median player. A reasonable guess for the optimal reward design is to (a)~pay nothing to the ranks below the target~$\alpha$ and (b) distribute the reward budget uniformly over the ranks above. We show in Theorem~\ref{th:design} that this guess is correct, for any value of the drift parameter. By contrast, the guess is wrong in the finite player game: for nonnegative drift, the general shape is correct, but the optimal cut-off point can be at a rank strictly below the target rank~$k$. For negative drift, the sharp cut-off can be replaced by a smoothed shape which also pays a small number of rewards of different sizes. (A full characterization of the optimal reward is only available for zero drift; cf.~\cite{NutzZhang.21a}.) Again, the mean field limit proves useful in allowing for a fuller description and a clearer result. On the other hand, knowing only the mean field limit may suggest an over-simplified picture for the finite player game.
One previous model where the optimal reward design problem was solved completely for both $n$-player and mean field setting, is the Poissonian game of~\cite{NutzZhang.19} where players control the jump intensity and are ranked according to their jump times. There, the optimal designs are more similar between the two settings; part~(a) of the above guess is always correct---the optimal reward has a sharp cut-off exactly at the target rank---though the shape over the ranks above the target is concave rather than being flat as in~(b). A related mean field game is considered in~\cite{BayCviZhang.19}, with diffusion instead of Poissonian dynamics. Both~(a) and~(b) turn out to be correct in the mean field setting. The $n$-player game is not tractable and its optimal design was not studied. In the light of the present work, one should not take for granted that the shape is analogous to the mean field limit.

\paragraph{Mean Field Approximation.}
In the above discussion, the mean field model is formulated directly as a game with infinitely many players. To connect this model rigorously with the $n$-player game, we show in Theorem~\ref{th:equilibConvergence} that for any given reward function, the unique $n$-player equilibrium for the induced reward converges to the mean field counterpart. Moreover, the value function of a player in the $n$-player game converges locally uniformly to the value function in the mean field game. This way of connecting the two models is classical in the mean field game literature starting with~\cite{LasryLions.06a}; see in particular \cite{Bardi.12, CardaliaguetDelarueLasryLions.15, CarmonaDelarueLacker.17, Fischer.14,Lacker.14, Lacker.18b}.
Another way of connecting the two models, going back to~\cite{HuangMalhameCaines.06}, is to fix the optimal strategy from the mean field equilibrium and consider it in the $n$-player game for large~$n$. Consistent with a broad literature (among others, \cite{CampiFischer.18,CarmonaDelarue.13,CarmonaDelarueLacker.17,CarmonaLacker.15,CecchinFischer.18}), we show in Theorem~\ref{th:epsEquilibrium} that for any continuous reward function, this control induces an $\eps$-Nash equilibrium for large~$n$; that is, players cannot improve their expected performance by more than~$\eps$ through unilateral deviations from the mean field strategy. Surprisingly,  continuity is necessary: any discontinuity in the reward function is shown to rule out the $\eps$-Nash equilibrium property for large~$n$. A discontinuity in reward leads to a gap in the support of the mean field equilibrium distribution. 
Due to a knife-edge phenomenon in the sampling for large but finite~$n$, a player can  improve substantially by unilaterally stopping inside the gap with a well-chosen distribution. In the study of diffusive mean field games with absorption, \cite[Section~7]{CampiFischer.18} described an example with degenerate volatility where the $\eps$-Nash equilibrium property fails. The degeneracy is exogenously chosen so that absorption cannot occur in the $n$-player game, but will occur in the mean field game, therefore creating a disconnect between the two. In~\cite{AnkirchnerKaziTaniWendtZhou.21}, on the other hand, the $\eps$-Nash equilibrium property always holds, despite the reward being discontinuous, because the dynamics of the game itself (nondegenerate volatility) guarantee a smooth equilibrium. The models of~\cite{CecchinDaiPraFischerPelino.18, DelarueFoguenTchuendom.18, Nutz.16,SanMartinNutzTan.18} highlighted a different type of discrepancy where some mean field equilibria can fail to be limits of $n$-player equilibria. Those examples arise due to non-uniqueness of mean field equilibria, thus are orthogonal to the issues in the present work.

Next, we discuss the quality of the mean field approximation for the reward design problem (Section~\ref{se:convOfOptDesign}); here using the mean field proxy seems particularly attractive because the optimal $n$-player design was fully solved only for zero drift. Our numerical discussion uses the zero drift case, for that same reason. We observe that the optimal design for the $n$-player problem converges to the mean field counterpart. Moreover, the induced performance of the former in the $n$-player game converges to the performance of the latter in the mean field game. This is consistent with~\cite{ElieMastroliaPossamai.19}, where the authors prove convergence of the optimal designs and induced performances for an example of their diffusive game---which is much more complex, yet smoother, than ours. But more importantly, and maybe surprisingly, the \emph{quality of the mean field proxy} from the point of view of the principal is strikingly poor in the present model (this aspect was not studied in~\cite{ElieMastroliaPossamai.19}): for moderate~$n$, the performance induced by the mean field optimizer is significantly inferior to the exact $n$-player optimal design. For large~$n$, the performance \emph{deteriorates} even further, eventually achieving only 50\% of the optimum. The tractability of the present model allows us to explain the reason for this phenomenon in detail. In the literature, mean field approximations are often applied in finite-player games without further analysis. The present study may offer the message that the quality of the approximation warrants consideration, especially when smoothness is not guaranteed, and that the mean field model can yield an oversimplified picture of the $n$-player game in some cases.

\section{Mean Field Equilibrium}\label{se:MFE}

In this section we define the mean field contest as a game with a continuum of players and prove that there exists a unique Nash equilibrium. It will be shown in Section~\ref{se:MFconvergence} that this equilibrium is indeed the limit of associated $n$-player games as $n\to\infty$. Throughout, we fix a \emph{reward function,} defined as a right-continuous and decreasing\footnote{Increase and decrease are understood in the non-strict sense in this paper.} function $R:[0,1]\rightarrow \R_+$ satisfying $R(0)>R(1-)=R(1)$. It will be shown in Remarks~\ref{rk:LC-necessary} and~\ref{rk:RC-necessary}, respectively, that left-continuity at the last rank is essential for uniqueness of the equilibrium whereas right-continuity is essential for existence.

Each infinitesimal player~$i$ privately observes a drifted Brownian motion $X_{t}^{i}=x_{0}+\mu t + \sigma W_{t}^{i}$ with absorption at $x=0$ and chooses a (possibly randomized) stopping time $\tau_{i}<\infty$. The initial value $x_{0}\in(0,\infty)$, drift $\mu\in\R$ and dispersion $\sigma\in(0,\infty)$ are identical across players whereas the Brownian motions are independent. Let $Y_{i}=X^{i}_{\tau_{i}}$ be the position at stopping. If $Y_{i}$ are i.i.d.\ across players with law~$F$, the empirical distribution of $(Y_{i})$ is a.s.\ equal to~$F$, by the Exact Law of Large Numbers (cf.\ Remark~\ref{rk:ELLN}). That is, if all players  choose the same stopping distribution, then the collection of players (deterministically) ends up distributed accordingly. Hence the rank of player~$i$ if she stops at $x$ while all other players stop according to distribution~$F$, is defined as $1-F(x)$.\footnote{We use the same symbol for the distribution and its cdf when there is no danger of confusion. Note that if $F$ has an atom at $x$, many players may share the same rank.} If $F$ does not have an atom at~$x$, meaning that there are no ties at this rank, she receives the reward $R(1-F(x))$. Otherwise she receives the average of $R(1-y)$ over $y\in [F(x-), F(x)]$, which is equivalent to splitting ties uniformly at random. Thus, writing
$$
g(y):=R(1-y),
$$
the payoff $\xi^F(x)$ for stopping at $x$ if all other players use $F$ is
\begin{equation}\label{eq:defxi}\xi^F(x)=
\begin{cases}
g(F(x)) & \text{ if } F(x)=F(x-),\\
\frac{1}{F(x)-F(x-)}\int_{F(x-)}^{F(x)} g(y)dy & \text{ if } F(x)>F(x-).
\end{cases}
\end{equation}

The set $\cF$ of distributions that are \emph{feasible}, i.e., can be attained by stopping $X^{i}$ with a randomized stopping time, is characterized through Skorokhod's embedding theorem and the scale function~$h$ of $X:=X^{i}$, as observed in~\cite{SeelStrack.13}.

\begin{lemma}\label{le:attainableDistrib}
  The set $\cF$ consists of all distributions $F$ on $[0,\infty)$ satisfying 
  $\int h dF=1$ if $\mu> 0$ and $\int h dF\le 1$ if $\mu\le 0$, respectively, where
\begin{equation}\label{eq:defh}
  h(x)=h_{x_0}(x)=\begin{cases}
  \frac{\exp\left(\frac{-2\mu x}{\sigma^{2}}\right)-1}{\exp\left(\frac{-2\mu x_{0}}{\sigma^{2}}\right)-1} & \text{ if } \mu\neq 0,\\
  \frac{x}{x_0} & \text{ if } \mu=0.
  \end{cases}
\end{equation}
\end{lemma} 

This result is due to to \cite{Hall.69}.\footnote{See~\cite[Section~9]{Obloj.04} for general background and a derivation. The extension to the present case with absorbing boundary is immediate.} We say that $F\in\cF$ is a mean field \emph{equilibrium} if no player is incentivized to deviate from $F$; that is, $\int \xi^F \,dF\geq \int \xi^F \,d\tF$ for all $\tF\in\cF$. The associated value function $u(x)$ is defined as the supremum expected reward achievable for a player starting at level $x$ (instead of $x_{0}$) if all others use~$F$.
Denote the average reward by $\bar R=\int_0^1 R(r)dr$ and set
\[
\bar \mu_\infty:=\frac{\sigma^2}{2x_0}\log\left(\frac{R(0)-R(1)}{R(0)-\bar R}\right)>0. \]
In the following result, $g^{-1}$ denotes the \emph{right}-continuous inverse
\[g^{-1}(z)=\inf\{y\in[0,1]: g(y)>z\}~~\mbox{for}~~z\in [R(1), R(0)), \qquad g^{-1}(R(0)):=1.\]

\begin{theorem}\label{thm:MFE}
Let $\mu<\bar\mu_\infty$. There exists a unique equilibrium. Its cdf is
\begin{equation}\label{eq:MFE}
F^\ast(x)=g^{-1}\left([R(1)+(\bar R-R(1)) h(x)] \wedge R(0)\right),
\end{equation}
where $g^{-1}$ is the right-continuous inverse of $g$, and the equilibrium value function is
$$u^*(x)=[R(1)+ (\bar R-R(1))h(x) ]\wedge R(0).$$ In particular, the equilibrium has compact support $[0,\bar x]$ for $\bar x = h^{-1}(\frac{R(0)-R(1)}{\bar R - R(1)})$ and its atoms are in one-to-one correspondence with intervals where the reward $R$ is constant.
\end{theorem}

\begin{remark}
The equilibrium distribution~\eqref{eq:MFE} is invariant under affine transformations of the reward~$R$. In particular, we may normalize the reward to satisfy $R(1)=0$ and $\bar R=1$ without loss of generality.
\end{remark}

\begin{remark}\label{rk:LC-necessary}
  The condition $R(1-)=R(1)$ is not necessary for the existence result in Theorem~\ref{thm:MFE} (it is not used in the proof), but it is crucial for uniqueness. Indeed, we claim that infinitely many equilibria arise whenever $R(1-)> R(1)$. To see this, fix a constant $R(1-)\geq\beta> R(1)$ and define the new reward $\tR$ by $\tR(1)=\beta$ and $\tR(r)=R(r)$ for $r<1$. We assume that $\beta$ is so that the constant $\bar\mu_{\infty}$ associated with~$\tR$ still satisfies~$\mu<\bar\mu_{\infty}$. As mentioned above, the reward~$\tR$ admits an equilibrium~$\tilde F^*$ as described Theorem~\ref{thm:MFE}, and inspection of the formula shows that $\tilde F^*$ differs from the equilibrium~$F^*$ corresponding to~$R$. More generally, $\tilde F^*$ is different for any two choices of~$\beta$. To prove the claim, we argue that $\tilde F^*$ is also an equilibrium for~$R$. If all players use the same stopping distribution, their value functions are the same under both rewards because achieving the last rank is a nullset for any player. However the rewards differ in the analysis of unilateral deviations: the inequality $\tR(1)> R(1)$ implies that if a player is not incentivized to deviate under~$\tR$, the same holds under~$R$. In particular, $\tilde F^*$ is also an equilibrium under~$R$, proving the claim. Conversely, $F^{*}$ need not be an equilibrium under~$\tR$, as can be seen from our uniqueness result for $\beta=R(1-)$.
\end{remark}

\begin{remark}\label{rk:ELLN}
  The framework of~\cite{Sun.06} allows for the rigorous construction of a continuum of (a.e.) independent processes satisfying an Exact Law of Large Numbers. A short summary of the pertinent results can be found, e.g., in \cite[Section~3]{Nutz.16}. Alternately to explicitly formulating the game with a continuum of players, one can also directly analyze the problem of a ``representative'' player facing a distribution, as it is sometimes done in the literature on mean field games---this corresponds to taking the Exact Law of Large Numbers as a given.
\end{remark} 

\section{Proof of Theorem~\ref{thm:MFE}}\label{se:MFEproof}

We first show by a direct verification argument that the stated distribution is indeed an equilibrium. The proof of uniqueness occupies the remainder of the section.

\begin{proof}[Proof of Theorem~\ref{thm:MFE}---Existence.]
  In view of 
  \begin{align*}
  \int h dF^*&= \int_0^{h(\infty)} (1-F^*\circ h^{-1}(w))dw
  = \int_0^{\frac{g(1)-g(0)}{\bar R-g(0)}} \left(1-g^{-1}(g(0)+(\bar R-g(0))w])\right) dw\\
  &= \frac{1}{\bar R-g(0)}\int_{g(0)}^{g(1)} \left(1-g^{-1}(z)\right) dz =\frac{\bar R-g(0)}{\bar R-g(0)}=1,
  \end{align*}
  Lemma~\ref{le:attainableDistrib} yields that $F^{*}\in\cF$. To see that $F^{*}$ is an equilibrium, fix some player~$i$ and suppose that all other players stop according to $F^\ast$. Using the property $g(g^{-1}(z))\leq z$ of the right-continuous inverse of the left-continuous function $g(y)=R(1-y)$, we have
  \[
    \xi^{F^*}(x)\leq g(F^*(x)) \leq [R(1)+(\bar R-R(1)) h(x)] \wedge R(0) =:\varphi(x).
  \]
  By It\^o's formula and Jensen's inequality, $\varphi(X_{t})$ is a bounded supermartingale. Hence, optional sampling  implies that for any finite stopping time $\tau$,
\[E[\xi^{F^\ast}(X_\tau)]\le E[\varphi(X_\tau)] \leq \varphi(x_{0})=\bar R.\]
On the other hand, player $i$ can attain $\bar R$ by choosing $F^\ast$, by symmetry. This shows that $F^\ast$ is optimal for player~$i$ and hence that $F^\ast$ is an equilibrium.
\end{proof}

\subsection{Relating Constant Rewards to Atoms, and Jumps in Reward to Gaps in Support}\label{se:constantRewardsAndAtoms}

We first relate atoms in equilibrium distributions to intervals of constancy of the reward (and hence of~$g$). Technical details aside, the message is that in equilibrium, equal pay must correspond to equal performance: ranks with the same reward are occupied by players that stop at the same level, and vice versa.

We do not yet impose the continuity properties of $R$, which will allow us to prove that they are important for the existence of equilibria. Instead, $R$ is any decreasing function in this subsection, which of course implies that its discontinuities are of jump-type.

\begin{lemma}\label{lemma:atom}
Let $F\in\cF$ be a mean field equilibrium. If $F$ has an atom at $x_1\in\R_+$, then $g$ is  constant on $(F(x_1-), F(x_1)]$. As a result, we have $\xi^F(x)=g(F(x))$ for all $x\in\R_{+}$. %
\end{lemma}

\begin{proof}
Set $y_0:=F(x_1-)$ and $y_1:=F(x_1)$. Let $\nu$ be the measure associated with~$F$. Consider for each $\eps>0$ the perturbed measure
\[\nu_\eps:=\lambda_\eps (\nu+(y_1-y_0)(\delta_{x_1+\eps}-\delta_{x_1}))+(1-\lambda_\eps)\delta_0,\]
where $\lambda_\eps\in (0,1)$ is chosen so that
\[\int h(x) d\nu_\eps(x)=\lambda_\eps\left(\int h(x)d\nu(x)+(y_1-y_0)[h(x_1+\eps)-h(x_1)]\right)=\int h(x)d\nu(x).\]
This ensures that $\nu_\eps\in \cF$.
Suppose that $g$ is not constant on $(y_0,y_1]$; then
\begin{align*}
\xi^F(x_1+\eps)-\xi^F(x_1) \ge g(y_1)-\frac{1}{y_1-y_0}\int_{y_0}^{y_1} g(y)dy=:\eta>0,
\end{align*}
where $\eta$ is clearly independent of $\eps$. 
This implies
\begin{align*}
&\int \xi^F(x) d\nu_\eps(x) - \int \xi^F(x)d\nu(x)\\
&=\lambda_\eps\left(\int \xi^F(x)d\nu(x)+(y_1-y_0)[\xi^F(x_1+\eps)-\xi^F(x_1)]\right)+(1-\lambda_\eps)\xi(0) - \int \xi^F(x)d\nu(x)\\
& = (1-\lambda_\eps)\left(\xi^F(0)-\int \xi^F(x)d\nu(x)\right)+\lambda_\eps (y_1-y_0)[\xi^F(x_1+\eps)-\xi^F(x_1)]\\
&\ge (1-\lambda_\eps)\left(\xi^F(0)-\int \xi^F(x)d\nu(x)\right)+\lambda_\eps (y_1-y_0)\eta.
\end{align*}
Using $\lim_{\eps \rightarrow 0+}\lambda_\eps =1$, $y_1-y_0>0$ and $\eta>0$, we obtain that 
$\int \xi^F dF_\eps-\int \xi^F dF>0$ for $\eps$ sufficiently small, contradicting the assumption that $F$ is an equilibrium.
Finally, if $g$ is constant on $(y_0, y_1]$, it is clear that $\xi^F(x_1)=g(y_1)=g(F(x_1))$.
\end{proof}

As $g$ is increasing, each level set $\{g=c\}$ is an interval. If the interval has positive length, we say that $g$ has a \emph{flat segment} at level $c$. In all that follows, we denote by $F^{-1}(y)=\inf\{x: F(x)\geq y\}$ the \emph{left}-continuous inverse (or quantile function) of~$F$.

\begin{lemma}\label{lemma:g-flat}
Suppose $g$ has a flat segment at level $c$, so that $y_1=\inf\{y\in[0,1]: g(y)=c\}$ and $y_2=\sup\{y\in[0,1]: g(y)=c\}$ satisfy $y_{1}<y_{2}$. Suppose $F\in\cF$ is a mean field equilibrium, define $x^+_1=F^{-1}(y_1+)$ and $x_2=F^{-1}(y_2)$. Then we must have $x^+_1=x_2$. Moreover, $F(x^+_1-)=y_1$ and $F(x^+_1)=y_2$.
\end{lemma}
\begin{proof}
Let $\nu$ be the measure associated with $F$ and $x_1^\eps=F^{-1}(y_1+\eps)$. Suppose on the contrary that $x_2>x_1^\eps$ for some $\eps\in (0,y_2-y_1)$. Then $y_1< F(x^\eps_1)<y_2$ and thus $\xi^F(x^\eps_1)=c$. Consider the measure 
\[\zeta:=\nu|_{(x^\eps_1,x_2)}+(y_2-F(x_2-))\delta_{x_2}\]
with total mass $|\zeta|=y_2-F(x^\eps_1)>0$. We distinguish two cases:

(i) Case $y_2<1$. In this case, let $\beta\in (y_2,1)$ and $x_\beta:=F^{-1}(\beta)\in [x_2,\infty)$. For some $\lambda\in[0,1]$ to be determined later, define the measure
\[\nu_\lambda=\nu-\zeta+|\zeta| (\lambda \delta_{x^\eps_1}+(1-\lambda)\delta_{x_\beta}).\]
In words, $\nu_\lambda$ is obtained from $\nu$ by removing all mass on $(x^\eps_1,x_2)$, plus possibly an additional atom at $x_2$ so that the total removed mass is $y_2-F(x^\eps_1)$, and moving this mass to atoms at $x^\eps_1$ and $x_\beta$ according to weights $\lambda$ and $1-\lambda$. Clearly $\nu_\lambda$ is a probability measure supported on $\R_+$, and we have
\begin{align*}
\int hd\nu_\lambda-\int h d\nu 
&=-\int h d\zeta+|\zeta| \left(\lambda h(x^\eps_1)+ (1-\lambda)h(x_\beta)\right)\\
&=\lambda \int  \left[h(x^\eps_1)-h\right] d \zeta + (1-\lambda) \int  \left[h(x_\beta)-h\right] d \zeta.
\end{align*}
In view of $\int  \left[h(x^\eps_1)-h\right] d \zeta<0$ and $\int  \left[h(x_\beta)-h\right] d \zeta\ge 0$, we can choose $\lambda\in [0,1)$ so that $\int hd\nu_\lambda=\int h d\nu$. We then have $\nu_\lambda \in\cF$ by Lemma~\ref{le:attainableDistrib}. Using the optimality of $F$ and that $\xi^F(x)=c$ for all $x\in[x^\eps_1,x_2)$, %
\begin{align*}
0&\ge \int \xi^F d\nu_\lambda-\int \xi^F d\nu
=\lambda \int \left[\xi^F(x^\eps_1)-\xi^F\right] d\zeta +(1-\lambda)\int \left[\xi^F(x_\beta)-\xi^F\right] d\zeta\\
&= \lambda (c-\xi^F(x_2))\zeta\{x_2\}+(1-\lambda) \left[\xi^F(x_\beta)-c\right] \nu(x^\eps_1,x_2) +(1-\lambda)(\xi^F(x_\beta)-\xi^F(x_2))\zeta\{x_2\}.
\end{align*}
Lemma~\ref{lemma:atom} rules out the possibility that $F(x_2-)<y_2<F(x_2)$, so we must be in one of the following two subcases:

\begin{itemize}
\item [(i-a)] $F(x_2-)=y_2$. In this case, $\zeta\{x_2\}=0$ and
\[\int \xi^F d (\nu_\lambda-\nu)=(1-\lambda) \left[\xi^F(x_\beta)-c\right] |\zeta|.\]
Using %
$|\zeta|>0$ and $\xi^F(x_\beta)\ge g(\beta)>c$ and $\lambda<1$, we obtain the contradiction that $\int \xi^F d (\nu_\lambda-\nu)>0$.

\item[(i-b)] $F(x_2-)<y_2=F(x_2)$. In this case, Lemma~\ref{lemma:atom} implies $\xi^F(x_2)=g(F(x_2))=g(F(x_2-)+)=c$ and we reach the same contradiction:
\begin{align*}
\int \xi^F d (\nu_\lambda-\nu) &=(1-\lambda) \left[\xi^F(x_\beta)-c\right] \nu(x^\eps_1,x_2)+(1-\lambda)(\xi^F(x_\beta)-c)\zeta\{x_2\}\\
&=(1-\lambda) \left[\xi^F(x_\beta)-c\right] |\zeta|>0.
\end{align*}
\end{itemize}

(ii) Case $y_2=1$. Then $y_1>0$ as $g$ is not a.e.\ constant. Let $x_1:=F^{-1}(y_1) \le x^\eps_1$. We note that $\nu[0,x_1]\ge y_1>0$ and consider the measure 
\begin{align*}
\nu_{\lambda}&=\nu-\lambda \zeta -(1-\lambda)\nu|_{[0,x_1]}+\left\{\lambda |\zeta|+(1-\lambda) \nu[0,x_1]\right\}\delta_{x^\eps_1}
\end{align*}
where $\lambda \in [0,1)$ is again chosen so that $\int h d\nu_\lambda= \int hd\nu$. We have
\begin{align*}
0 \ge \int \xi^F d\nu_\lambda-\int \xi^F d\nu
&=\lambda \int \left[\xi^F(x^\eps_1)-\xi^F\right] d\zeta +(1-\lambda)\int \left[\xi^F(x^\eps_1)-\xi^F\right] d\nu|_{[0,x_1]}\\
&= \lambda (c-\xi^F(x_2))\zeta\{x_2\} +(1-\lambda)\int \left[c-\xi^F\right] d\nu|_{[0,x_1]}.
\end{align*}
Similarly as in Case~(i), one can show that either $\zeta\{x_2\}=0$ or $\xi^F(x_2)=c$, both of which lead to
\[0\ge (1-\lambda)\int \left[c-\xi^F\right] d\nu|_{[0,x_1]}\ge 0\]
and thus
$\int \left[c-\xi^F(x)\right] d\nu|_{[0,x_1]}=0.$
It follows that $\xi^F(x)=c$ for $\nu$-a.e.\ $x\in [0,x_1]$. On the other hand, the definitions of $x_1$ and $y_1$ imply that $F(x)<y_1$ and $\xi^F(x)<c$ for all $x<x_1$. So it must hold that either $x_1=0$ or $\nu[0,x_1)=0$. Both cases lead to $F(x_1-)=0$ and $\nu\{x_1\}\ge y_1>0$. But $F(x_1-)=0$ yields, by Lemma~\ref{lemma:atom}, that $\xi^F(x_1)=g(F(x_1-)+)=g(0+)<c$, whereas $\nu\{x_1\}>0$ implies 
$\xi^F(x_1)=c$, a contradiction. This completes the proof that $F^{-1}(y_1+)=F^{-1}(y_2)$.

Finally, let $x^+_1:=F^{-1}(y_1+)$. Clearly $F(x^+_1)=F(F^{-1}(y_2))\ge y_2$. For any $x<x^+_1$ and $\eps>0$, we have $x< x^\eps_1$ and $F(x)<y_1+\eps$. Passing to the limit then yields $F(x^+_1-)\le y_1<y_2$. As $g$ is constant on $(F(x^+_1-), F(x^+_1)]$, we must have $F(x^+_1-)=y_1$ and $F(x_1^+)=y_2$.
\end{proof}

\begin{remark}\label{rk:g-flat}
If in Lemma~\ref{lemma:g-flat} we also have $g(y_1)=c$, then the proof goes through with $x^\eps_1$ replaced by $x_1=F^{-1}(y_1)\vee 0$. (The reason for using $x_1^\eps$ is to have $\xi^F(x_1^\eps)=g(F(x_1^\eps))=c$.) As a result, we have $x_1=x^+_1=x_2$. In particular, if $y_1=0$, then $F(0)=y_2$.
\end{remark}

\begin{remark}\label{rk:RC-necessary}
Lemmas~\ref{lemma:atom} and~\ref{lemma:g-flat} imply that for the existence of a mean field equilibrium, it is necessary that $g$ be left-continuous at any level for which it contains a flat segment. In particular, if the reward function $R$ is piecewise constant, \emph{a mean field equilibrium can only exist if $R$ is right-continuous.}
\end{remark}

The feasibility constraint yields one equation to pin down the equilibrium. The best way to illustrate this is to go through a particular case of Theorem~\ref{thm:MFE} where the reward function is of cut-off type. That is the purpose of the next proposition---here, the feasibility constraint and the preceding results on atoms are already sufficient to uniquely identify the equilibrium.

\begin{proposition}\label{prop:MFG-cutoff}
Let $R(r)=\frac{1}{\alpha} 1_{[0,\alpha)}(r)$ for some $\alpha\in (0,1)$. Then the unique mean field equilibrium is given by the two-point distribution $\nu_\alpha:=(1-\alpha)\delta_0+\alpha \delta_{x_1}$ where $x_1$ is the unique point in $(x_0,\infty)$ with $h(x_1)=1/\alpha$.
\end{proposition}

\begin{proof}
We first derive a necessary condition for $F\in \cF$ to be an equilibrium. Let $x_\alpha=F^{-1}((1-\alpha)+)$. 
By Lemma~\ref{lemma:g-flat} and Remark~\ref{rk:g-flat}, we have $F(0)=1-\alpha$, $F(x_\alpha-)=1-\alpha$ and $F(x_\alpha)=1$. That is, the measure associated with $F$ must take the form
$\nu=(1-\alpha)\delta_0+\alpha \delta_{x_\alpha}$.
To determine $x_\alpha$, we first note that $x_\alpha\le x_1$, for otherwise $\int h d\nu > \int h d\nu_\alpha=1$, contradicting $\nu\in\cF$. Suppose $x_\alpha<x_1$ (which is only feasible if $\mu\le 0$), then there is $\beta\in (\alpha,1)$ such that $\nu'=(1-\beta)\delta_0+\beta \delta_{x_\alpha}$
is feasible. In view of  $\xi^F(x_\alpha)=g(F(x_\alpha))=g(1)>g(1-\alpha)=\xi^F(0)$, the distribution $\nu'$ is strictly preferable to $\nu$ when the other players choose $\nu$. As a result, $x_\alpha=x_1$, which uniquely identifies~$F$.
To check that $F$ is indeed an equilibrium, we argue as in the beginning of Section~\ref{se:MFEproof}.
\end{proof}

The above proof does not generalize to piecewise constant reward functions with multiple jumps: while the feasibility constraint still yields one equation, there are now multiple unknowns (the locations of the atoms). To determine mean field equilibria for general reward functions, it is necessary to analyze the effect of jumps in some detail. 
Let 
$$%
J(g):=\{y\in(0,1): g(y-)<g(y+)\}
$$%
be the set of interior jump points of $g$. The next lemma says that any jump of~$g$---or equivalently of $R$---induces a flat segment in any equilibrium distribution. (The reasoning  in Remark~\ref{rk:LC-necessary} shows that this assertion fails at $y=0$, whence the definition of~$J(g)$ considers only interior jumps.)

\begin{lemma}\label{lemma:g-jump}
Let $F$ be a mean field equilibrium. For each $y\in J(g)$, the interval $\{x\ge 0: F(x)=y\}$ has positive length.
\end{lemma}

\begin{proof}
Let $y_1\in J(g)$, then $x_1:=F^{-1}(y_1)\in [0,\infty)$ as $y_1\in (0,1)$. 
Suppose for contradiction that $\{x\ge 0: F(x)=y_1\}$ has zero length, then $F(x)>y_1$ for all $x>x_1$. Let $\nu$ be the measure associated with $F$. In the remainder of the proof we construct a feasible distribution $\nu'$ that is strictly better than~$\nu$. 
By Lemma~\ref{lemma:atom} we have either $F(x_1-)=y_1$ or $F(x_1-)<y_1=F(x_1)$.

(i) Case $F(x_1-)=y_1$. In this case, $x_1>0$ and $F$ is non-constant in any left neighborhood of $x_{1}$. Fix $\gamma\in (0,g(y_1+)-g(y_1-))$ and observe that
\[\lim_{\eps\rightarrow 0+} \xi^F(x_1-2\eps)=g(y_1-), \quad\quad \lim_{\eps\rightarrow 0+} \frac{h(x_1-\eps)-h(x_1-2\eps)}{h(x_1)-h(x_1-2\eps)}=\frac{1}{2}.\]
We can thus find $\eps>0$ such that 
\begin{equation}\label{lemma:jump:eq1}
\xi^F(x_1-2\eps)>g(y_1-)-\gamma/2
\end{equation}
and
\begin{equation*}
\frac{h(x_1-\eps)-h(x_1-2\eps)}{h(x_1)-h(x_1-2\eps)}>\frac{\gamma}{g(y_1+)-g(y_1-)+\gamma}.
\end{equation*}
The measure
$\zeta:=\nu |_{(x_1-\eps,x_1)}$ has mass $|\zeta|>0$. Consider the probability measure
\[\nu':=\nu-\zeta+|\zeta|\left(\lambda \delta_{x_2}+(1-\lambda)\delta_{x_1-2\eps}\right)\]
where $x_2\in (x_1,\infty)$ is chosen to satisfy 
\begin{equation}\label{lemma:jump:eq2}
\frac{h(x_1-\eps)-h(x_1-2\eps)}{h(x_2)-h(x_1-2\eps)}>\frac{\gamma/2}{g(y_1+)-g(y_1)+\gamma/2}
\end{equation}
and
\begin{equation}\label{lemma:jump:eq3}
\lambda:=\frac{\frac{1}{|\zeta|}\int hd\zeta -h(x_1-2\eps)}{h(x_2)-h(x_1-2\eps)}\in\left(\frac{h_1(x_1-\eps)-h(x_1-2\eps)}{h(x_2)-h(x_1-2\eps)}, 1\right).
\end{equation}
It is easy to check that $\int h d(\nu'-\nu)=0$, hence $\nu'\in\cF$ by Lemma~\ref{le:attainableDistrib}. To see that $\nu'$ is strictly better than $\nu$, we use \eqref{lemma:jump:eq1}--\eqref{lemma:jump:eq3} and $F(x_2)>y_1$:
\begin{align*}
\int \xi^F & d(\nu'-\nu)
=\lambda \int (\xi^F(x_2)-\xi^F)d\zeta+(1-\lambda)\int (\xi^F(x_1-2\eps)-\xi^F)d\zeta\\
&\ge \lambda (g(y_1+)-g(y_1-)) |\zeta| +(1-\lambda)(\xi^F(x_1-2\eps)-g(y_1-))|\zeta|\\
&> |\zeta|\left( \lambda [g(y_1+)-g(y_1-)] -(1-\lambda)\frac{\gamma}{2}\right)
= |\zeta|\left( \lambda [g(y_1+)-g(y_1-)+\gamma/2] -\frac{\gamma}{2}\right)>0.
\end{align*}

(ii) Case $F(x_1-)<y_1=F(x_1)$. In this case, $\{x\ge 0: F(x)=y_1\}$ having zero length implies that $F$ is non-constant in any right neighborhood of $x_{1}$. Moreover, Lemma~\ref{lemma:atom} implies that $g(y_1-)=g(y_1)$.
 Fix $\gamma\in (0, (y_1-F(x_1-))(g(y_1+)-g(y_1)))$ and
$\eps>0$ such that $\xi^F(x_1+\eps)<g(y_1+)+\gamma$.
We define $\zeta:=\nu |_{[x_1,x_1+\eps)}$ and 
\[\nu':=\nu-\zeta+|\zeta|\delta_{x_2},\]
where $x_2>x_1$ is to be determined. Since $h$ is strictly increasing and $\nu(x_1, x_1+\eps)>0$, we see that $|\zeta| h(x_1)<\int h d\zeta <|\zeta| h(x_1+\eps)$ and consequently there exists $x_2\in (x_1, x_1+\eps)$ such that 
\[\int h d(\nu'-\nu)= |\zeta| h(x_2)-\int h d\zeta=0.\]
For this choice of $x_2$, we have $\nu'\in\cF$ by Lemma~\ref{le:attainableDistrib}. Moreover, $\nu'$ is strictly better than $\nu$:
\begin{align*}
\int \xi^F d(\nu'-\nu)
&=(y_1-F(x_1-))[\xi^F(x_2)-\xi^F(x_1)]+\int (\xi^F(x_2)- \xi^F)d\nu|_{(x_1,x_1+\eps)}\\
&\ge (y_1-F(x_1-))(g(y_1+)-g(y_1))+(g(y_1+)-\xi^F(x_1+\eps))\nu(x_1,x_1+\eps)\\
&\ge  (y_1-F(x_1-))(g(y_1+)-g(y_1))-\gamma>0
\end{align*}
by the choice of $\gamma$.
\end{proof}

\subsection{Characterizing the Equilibrium}\label{se:uniquenessByMinimization}

From now on, we shall work under the assumption that $\mu<\bar \mu_\infty$ and $R$ is right-continuous with $R(0)>R(1-)=R(1)$.

The general idea of the uniqueness argument is to analyze a minimization problem: in equilibrium, the opposing players act such as to minimize the value function of a given representative player, subject to the constraint that the opponents act symmetrically. This turns out to be substantially more involved than in the $n$-player case, due to the possible presence of  atoms in the equilibrium distribution and the non-invertibility of the function~$g$. 

Let $u^F$ be the value function of a representative player if the other players use $F\in\cF$.
Dynamic programming and optimal stopping theory yield
\[
u^F(x_0)=\sup_{\tau<\infty} E[\xi^F(X_\tau)]=(\xi^F \circ h^{-1})^{conc}(1)\le (g\circ F \circ h^{-1})^{conc}(1)
\]
where $h$ is the scale function~\eqref{eq:defh} with normalization $h(x_{0})=1$ and \emph{conc} denotes the concave envelope on $\mathbb{R}_+$. 
The last inequality is due to possible breaking of ties, cf.~\eqref{eq:defxi}. Lemma~\ref{lemma:atom} shows that the inequality must be an equality if~$F$ is a mean field equilibrium, even if ties do occur.
On the other hand, if $F^\ast$ is an equilibrium, we must have 
\begin{equation}\label{minimization1b}
  \bar{R}=u^{F^\ast}(x_{0}) = \min_{F\in\cF} u^{F}(x_{0}).
\end{equation}
Indeed, given arbitrary $F\in\cF$, a representative player can achieve $\bar{R}$ by also choosing~$F$, and in equilibrium, this is the best possible performance, by symmetry. Combining the two arguments, any mean field equilibrium $F^\ast$ must satisfy
\[(g\circ F^\ast\circ h^{-1})^{conc}(1)=u^{F^\ast}(x_0)=\min_{F\in\cF}u^F(x_0)\le \min_{F\in\cF} (g\circ F \circ h^{-1})^{conc}(1).\]
That is, 
\begin{equation*}\label{eq:Phi}
F^\ast\in \argmin_{F\in\cF} \Phi(F)^{conc}(1), \quad \mbox{where}\quad \Phi(F):=g\circ F\circ h^{-1}.
\end{equation*}
We also write $\Phi^{-1}(\phi):=g^{-1}\circ \phi\circ h$. We recall that $g^{-1}$ denotes the right-continuous inverse of $g$; in particular, 
 $g^{-1}(g(y))\ge y$ 
and $g(g^{-1}(z))\le z$. Similarly, $\Phi^{-1}(\Phi(F))\ge F$ and $\Phi(\Phi^{-1}(\phi))\le \phi$. 
Finally, we denote 
\begin{align*}
 \bar w_F&:=\inf \{w\in [0, h(\infty)]: \Phi(F)(w)=g(1)\}\le h(\infty),\\
 \bar y&:=\inf\{y\in [0,1]: g(y)=g(1)\},\\
 \cF'&:=\{F\in \cF: \bar w_F>1, \text{ and } \bar w_F<h(\infty) \mbox{ in case $\bar y<1$}\}.
\end{align*}

\begin{lemma}\label{lemma:MFG4}
If $F$ is a mean field equilibrium, then $F\in\cF'$ and $\bar w_F=\inf\{w\in [0, h(\infty)]: F\circ h^{-1}(w)=1\}$. 
In particular, $F\circ h^{-1}(\bar w_F)=1$.
\end{lemma}
\begin{proof}
We first show $F\in\cF'$. Suppose $\bar w_F\le 1$. Then 
\[\Phi(F)^{conc}(1)\ge \Phi(F)^{conc}(\bar w_F)=\Phi(F)^{conc}(\bar w_F+)=g(1).\]
Consider the distribution $G(x)=\lambda 1_{[0,x_0+\eps)}(x)+ 1_{[x_0+\eps,\infty)}(x)$ where $\eps >0$ and $\lambda\in (0,\bar y)$ are chosen so that $\int h dG=(1-\lambda) h(x_0+\eps)=1$. We have $\Phi(G)=g(\lambda)1_{[0,h(x_0+\eps))}+g(1) 1_{[h(x_0+\eps),\infty)}$. The concave hull of this function is readily determined and in view of $g(\lambda)<g(1)$, we arrive at $\Phi(G)^{conc}(1)<g(1)\le \Phi(F)^{conc}(1)$, contradicting the optimality of $F$.

Suppose $\bar y<1$ and $\bar w_F=h(\infty)$. Then for all $x<\infty$, we have $h(x)<h(\infty)=\bar w_F$, which implies $g(1)>\Phi(F)(h(x))=g(F(x))$. But then $F(x)\le \bar y<1$ for all $x\in\R$, contradicting that $F$ is the cdf of a probability measure on $\R$.

We next show $F\circ h^{-1}(\bar w_F)=1$.  This is trivial if $\bar w_F=h(\infty)$, so we may assume that $\bar w_F<h(\infty)$. For any $w>\bar w_F$, we have $g(F\circ h^{-1}(w))=\Phi(F)(w)=g(1)$, which implies: 

(i) $F\circ h^{-1}(w)=1$ if $\bar y=1$. Then by right-continuity, $F\circ h^{-1}(\bar w_F)=1$.

(ii) $F\circ h^{-1}(w)>\bar y$ if $\bar y<1$ and $g(\bar y)<g(1)$. In this case, $h^{-1}(w) \ge F^{-1}(\bar y+)$ for all $w>\bar w_F$, which further yields $h^{-1}(\bar w_F)\ge F^{-1}(\bar y+)$. By Lemma~\ref{lemma:g-flat}, $F$ jumps from $\bar y$ to $1$ at $F^{-1}(\bar y+)$. It follows that $F\circ h^{-1}(\bar w_F)\ge F(F^{-1}(\bar y+))=1$.

(iii) $F\circ h^{-1}(w)\ge \bar y$ if $\bar y<1$ and $g(\bar y)=g(1)$. In this case, we use Remark~\ref{rk:g-flat} to obtain $h^{-1}(w)\ge F^{-1}(\bar y)=F^{-1}(\bar y+)$ and thus $h^{-1}(\bar w_F)\ge F^{-1}(\bar y+)$. We obtain the same conclusion as in~(ii).

Finally, for $w<\bar w_F$, $\Phi(F)(w)<g(1)$ implies $F\circ h^{-1}(w)<1$.
\end{proof}

\begin{lemma}\label{lemma:MFG3}
Let $F\in\cF'$.
Suppose there exists an increasing concave function $\phi\ge \Phi(F)$ on $[0,h(\infty)]$ satisfying $\phi(1)\le \Phi(F)^{conc}(1)$ and 
\begin{align*}
\int_0^{h(\infty)} (1-g^{-1}\circ \phi(w))dw &<1. %
\end{align*}
Then there exists $F'\in\cF$ such that $\Phi(F')^{conc}(1)<\Phi(F)^{conc}(1)$ and consequently, $F$ cannot be a mean field equilibrium.
\end{lemma}
\begin{proof}
Let $\phi$ be as stated. Note that $\bar w_F>1$ implies $\Phi(F)(1+)<g(1)$ which further yields 
$\Phi(F)^{conc}(1)<g(1)=\Phi(F)(\bar w_F+)\le \phi(\bar w_F+)=\phi(\bar w_F)$. Let $\bar w_\phi:=\inf\{w\ge 0: \phi(w)=g(1)\}$. Since $\phi(1)<g(1)$ and $\phi(\bar w_F+)=g(1)$, we know $1<\bar w_\phi \le \bar w_F$.
Consider four cases:

(i) $\mu\le 0$ and $\bar y=1$. In this case, $h(\infty)=\infty$ and $g^{-1}$ is continuous at $g(1)$. Choose $\eps\in (0,1)$ such that $\phi(\eps)<\phi(1)$. Such $\eps$ exists: as $\phi$ is increasing and concave, it must be strictly increasing before reaching $g(1)$.
Let $\phi_\eps(w):=\phi(\eps w)$. Then $\phi_\eps$ is concave on $\R_+$ and satisfies $\phi_\eps(1)<\phi(1)$.
Next, define $F_\lambda:=\Phi^{-1}(\lambda \phi+(1-\lambda)\phi_\eps)$. One can check that $F_\lambda$ is right-continuous and satisfies $F_\lambda(\infty)=1$. We also have that for $\lambda\in[0,1)$, $\Phi(F_\lambda)^{conc}(1)\le \left(\lambda \phi +(1-\lambda) \phi_\eps \right)^{conc}(1)=\lambda \phi(1) +(1-\lambda) \phi_\eps(1)<\phi(1)\le \Phi(F)^{conc}(1)$, showing that $F_\lambda$ is strictly better than $F$. To reach the desired contradiction, it remains to show the feasibility of $F_\lambda$ for $\lambda$ sufficiently close to one. 
We have
\begin{align*}
\int h dF_\lambda =\int_0^\infty (1-F_\lambda \circ h^{-1})(w) dw
=\int_0^\infty (1-g^{-1}\circ (\lambda \phi +(1-\lambda)\phi_\eps))(w)dw.
\end{align*}
As $g^{-1}$ is monotone, it has at most countably many points of discontinuity, and $\bar y=1$ implies that $g(1)$ is not one of them. For any $z<g(1)$, the set $\{w\ge 0: \phi(w)=z\}$ has zero Lebesgue measure because $\phi$ is strictly increasing before reaching $g(1)$. It follows that as $\lambda\rightarrow 1$, the integrand converges a.e.\ to $1-g^{-1}\circ \phi$. Using
$0\le 1-g^{-1}\circ (\lambda \phi +(1-\lambda)\phi_\eps) \le 1-g^{-1}\circ \phi_\eps  \le  1-g^{-1}\circ \Phi(F)(\eps\cdot id)\le 1-F \circ h^{-1} (\eps\cdot id)$ and $\int_0^\infty (1-F \circ h^{-1} (\eps w)) dw =\frac{1}{\eps}\int h dF<\infty$, dominated convergence yields that
\begin{align*}
\lim_{\lambda\rightarrow 1}\int h dF_\lambda&=\int_0^\infty (1-g^{-1}\circ \phi)(w)dw <1.
\end{align*}
By Lemma~\ref{le:attainableDistrib}, this shows that $F_{\lambda}$ is feasible for $\lambda$ sufficiently close to one. 

(ii) $\mu\le 0$ and $\bar y<1$. In this case, $1<\bar w_\phi\le \bar w_F<h(\infty)=\infty$. 
Choose $\eps>0$ such that $\bar w_\phi+\eps < h(\infty)$ and $\int_0^{\bar w_\phi} (1-g^{-1}\circ \phi)(w)dw + (1-\bar y)\eps<1$.
Let $\phi'_\eps$ denote the line connecting $(0,\phi(0))$ and $(\bar w_\phi+\eps, g(1))$ and capped at level $g(1)$; i.e.,
\begin{equation*}\label{eq:phi_eps}
\phi'_\eps (w)=\left(\phi(0)+\frac{g(1)-\phi(0)}{\bar w_\phi+\eps} w\right)\wedge g(1).
\end{equation*}
Then $\phi'_\eps$ is concave on $\R_+$ and satisfies $\phi'_\eps(1)<\phi(1)$. As in the previous case, we define 
$F'_\lambda:=\Phi^{-1}(\lambda \phi+(1-\lambda)\phi'_\eps)$. Then $F'_\lambda$ is a cdf supported on $\R_+$ which satisfies $\Phi(F'_\lambda)^{conc}(1)<\Phi(F)^{conc}(1)$ for all $\lambda\in[0,1)$. To check the feasibility of $F'_\lambda$ for $\lambda$ close to one, we write
\begin{align*}
\int h dF'_\lambda
& = \int_0^{\bar w_\phi} (1-g^{-1}(\lambda \phi(w) +(1-\lambda)\phi_\eps (w)))dw \\
&\quad + \int_{\bar w_\phi}^{\bar w_\phi+\eps}  (1-g^{-1} (\lambda g(1) +(1-\lambda)\phi_\eps (w)))dw.
\end{align*}
Using that $\phi$ is strictly increasing on $[0, \bar w_\phi]$, we obtain by bounded convergence  that
$%
\lim_{\lambda\rightarrow 1}\int h dF'_\lambda=\int_0^{\bar w_\phi} (1-g^{-1}\circ \phi)(w)dw + (1-\bar y)\eps <1.
$%

(iii) $0<\mu<\bar \mu_\infty$ and $\bar y=1$. In this case, $\bar w_\phi\le h(\infty)<\infty$ and $g^{-1}$ is continuous at $g(1)$. 
Let $\ell$ be the line segment connecting $(0,g(0))$ and $(h(\infty), g(1))$. 
We have
\begin{align*}
&\int h d\Phi^{-1}(\ell)= \int_0^{h(\infty)} \left(1-\Phi^{-1}(\ell)\circ h^{-1}(w)\right) dw=\int_0^{h(\infty)} \left(1-g^{-1}\circ \ell(w)\right)dw\\
&=\frac{h(\infty)}{g(1)-g(0)}\int_{g(0)}^{g(1)} \left(1-g^{-1}(y)\right) dy
=\frac{h(\infty)}{g(1)-g(0)}\left(\int_0^1 g(y)dy - g(0)\right)=\frac{h(\infty)}{h(\bar x)}>1.
\end{align*}
Since $\phi(0)\ge g(0)=\ell(0)$ and $\phi(h(\infty))=g(1)=\ell(h(\infty))$, by concavity, either $\phi> \ell$ on $(0, h(\infty))$ or $\phi= \ell$. The latter case is impossible as $F \le \Phi^{-1}(\Phi(F))\le \Phi^{-1}(\phi)=\Phi^{-1}(\ell)$ would imply $F \notin \cF$.
Set $F''_\lambda:=\Phi^{-1}(\lambda \phi+(1-\lambda)\ell)$. We again have $\Phi(F''_\lambda)^{conc}(1)\le \lambda \phi(1)+(1-\lambda)\ell(1)<\phi(1)\le \Phi(F)^{conc}(1)$ if $\lambda \in [0,1)$. Let 
\begin{align*}
I(\lambda):=\int h dF''_\lambda 
&=\int_0^{h(\infty)} (1-g^{-1} (\lambda \phi(w)+(1-\lambda)\ell(w)) )dw.
\end{align*}
Using the continuity of $g^{-1}$ at $g(1)$, the strict monotonicity of $\lambda \phi+(1-\lambda)\ell$ before reaching $g(1)$, and bounded convergence theorem, 
we deduce that $I(\cdot)$ is continuous on $(0,1)$ satisfying $I(1-)=\int_0^{h(\infty)} (1-g^{-1}\circ \phi(w))dw<1$ and 
\begin{align*}
I(0+)&=\int_0^{h(\infty)} (1-g^{-1} \circ \ell)(w)dw=\int_0^{h(\infty)} (1-\Phi^{-1}(\ell) \circ h^{-1})(w)dw
=\int h d\Phi^{-1}(\ell)>1.
\end{align*}
We may thus choose $\lambda_0\in(0,1)$ such that $I(\lambda_0)=1$. Then $F''_{\lambda_0}\in \cF$ and the contradiction is complete.

(iv) $0<\mu<\bar \mu_\infty$ and $\bar y<1$. In this case, $\bar w_\phi<h(\infty)<\infty$. Let $\phi'_\eps$ and $F'_\lambda$ be constructed as in Case~(ii) with $\eps$ satisfying $\bar w_\phi+2\eps<h(\infty)$. Define
\[G(x):=\begin{cases}
\gamma F'_\lambda(x), & x< h^{-1}(\bar w_\phi+2\eps),\\
1, & x\ge h^{-1}(\bar w_\phi+2\eps),
\end{cases}\]
for some $\gamma\in(0,1)$ to be determined.
We have $F'_\lambda(h^{-1}(\bar w_\phi+2\eps)-)\ge F'_\lambda(h^{-1}(\bar w_\phi+\eps))=g^{-1}(\lambda \phi(\bar w_\phi+\eps)+(1-\lambda)\phi'_\eps(\bar w_\phi+\eps))=1$ and
\[\int hdG=\gamma \int hdF'_\lambda +(1-\gamma) (\bar w_\phi+2\eps).\] 
In view of  $\int hdF'_\lambda<1$ and $\bar w_\phi+2\eps>1$, we can find $\gamma\in (0,1)$ such that $\int hdG=1$, and then $G$ is feasible. We arrive that the desired contradiction after noting that $\Phi(G)^{conc}(1)\le \Phi(F'_\lambda)^{conc}(1)<\Phi(F)^{conc}(1)$.
\end{proof}

\begin{lemma}\label{lemma:MFG5}
Let $F$ be a mean field equilibrium. Define
\[A=\bigcup_{y\notin J(g)} A_y, \quad \text{where}\quad A_y=\{w\in [0, h(\infty)]\cap\R: F\circ h^{-1}(w)=y\},\]
as well as $w_y=\inf A_y$. Then $w_y<\bar w_F$ for all $y\in J(g)$. Moreover, there exists a strictly increasing affine function $\ell_1\ge \Phi(F)$ satisfying $\ell_1(1)=\Phi(F)^{conc}(1)$ and %
\begin{itemize}
\item[(i)] $\ell_1(\bar w_F)=g(1)$,
\item[(ii)] $\Phi(F)=\ell_1\wedge g(1)$ on $A$,
\item[(iii)] $\Phi(F)(w_y)=\ell_1(w_y)$ for all $y\in J(g)$.
\end{itemize}
\end{lemma}

\begin{proof}
Let $F$ be an equilibrium, then $F\in\cF'$ by Lemma~\ref{lemma:MFG4}. Define $\phi:=\Phi(F)^{conc}$ and $\psi:=\ell_1\wedge g(1)$, where $\ell_1$ is an affine function passing through $(1,\phi(1))$ whose slope lies in the super-differential of $\phi$ at $w=1$. We have $\ell_1\ge \psi\ge \phi\ge \Phi(F)$, $g^{-1}\circ \psi\ge g^{-1}\circ \Phi(F)\ge F\circ h^{-1}$ and $\psi(1)\le \ell_1(1)=\phi(1)=\Phi(F)^{conc}(1)$. 
By Lemma~\ref{lemma:MFG4}, $\Phi(F)(1+)<g(1)=\Phi(F)(\bar w_F+)\le \phi(\bar w_F)\le g(1)$, which implies that $\phi(1)<g(1)=\phi(\bar w_F)$. As $\phi$ is increasing and concave, it must be strictly increasing before reaching level $g(1)$. Consequently, $\ell_1$ has positive slope.
For any $y\in J(g)$, $A_y$ has positive length by Lemma~\ref{lemma:g-jump}. Since $F\circ h^{-1}=y<1$ on $A_y$, we must have $A_y\subseteq [0,\bar w_F)$ and $w_y< \bar w_F$. 
It remains to show properties (i)--(iii). Specifically, we show below that if one of these properties does not hold, then $\int_0^{h(\infty)} (1-g^{-1}\circ \psi(w))dw<1$. Applying Lemma~\ref{lemma:MFG3} with $\psi$ being the increasing concave function, this contradicts that~$F$ is an equilibrium.

(i) Let $\bar w_1:=\ell_1^{-1}(g(1))$. Since $\ell_1\ge \Phi(F)$, we necessarily have $\bar w_1\le \bar w_F$.
Suppose $\bar w_1<\bar w_F$, then $F\circ h^{-1}<1$ in a right neighborhood of $\bar w_1$. Together with $g^{-1}\circ \psi\ge F\circ h^{-1}$%
, we obtain
\begin{align*}
&\int_0^{h(\infty)} (1-g^{-1}\circ \psi(w))dw=\int_0^{\bar w_1} (1-g^{-1}\circ \psi(w))dw \\
 &\le \int_0^{\bar w_1} (1-F\circ h^{-1}(w))dw< \int_0^{h(\infty)} (1-F\circ h^{-1}(w))dw= \int h dF\le 1.
\end{align*}

(ii) Suppose $\Phi(F)(w_0)<\psi(w_0)$ for some $w_0\in A$. As $F\circ h^{-1}(w_0)\notin J(g)$, we have $F\circ h^{-1}(w_0)\le g^{-1}(\Phi(F)(w_0))<g^{-1}(\psi(w_0))$. By the right-continuity of $F\circ h^{-1}$ and $g^{-1}\circ \psi$, it follows that $F\circ h^{-1}<g^{-1}\circ \psi$ in a right neighborhood of $w_0$. 
Thus
\begin{align*}
\int_0^{h(\infty)} (1-g^{-1}\circ \psi(w))dw &< \int_0^{h(\infty)} (1-F\circ h^{-1}(w))dw= \int h dF\le 1.
\end{align*}

(iii) Let $y\in J(g)$. 
Suppose $\Phi(F)(w_y)<\ell_1(w_y)$.
Let $w_y':=\ell_1^{-1}(\Phi(F)(w_y))$. We have $w'_y<\ell_1^{-1}(\ell_1(w_y))= w_y < \bar w_F$ and $\ell_1(w'_y)=\Phi(F)(w_y)=g(y)<g(y+)\le g(1)$. 
Define 
\[F_1(x):=
\begin{cases}
F(x), & \text{if } x<h^{-1}(w'_y) \text{ or } x\ge h^{-1}(w_y),\\
F(h^{-1}(w_y))=y, & \text{if } h^{-1}(w'_y)\le x < h^{-1}(w_y).
\end{cases}\]
Clearly, $F_1\ge F$. For $x\in [h^{-1}(w'_y), h^{-1}(w_y))$, we have $h(x)\in [w'_y,w_y)$ and $F(x)=F\circ h^{-1}(h(x))<y=F_1(x)$ by the definition of $w_y$. It follows that $\int hdF_1<\int hdF\le 1$. 
Now, observe that
\[\psi(w)\ge \Phi(F_1)(w)=
\begin{cases}
\Phi(F)(w), & \text{if } w<w'_y \text{ or } w\ge w_y,\\
g(y), & \text{if } w'_y\le w < w_y.
\end{cases}\]
This implies
\begin{align*}
& \int_0^{h(\infty)} (1-g^{-1}\circ \psi(w))dw \le \int_0^{h(\infty)} (1-g^{-1}\circ \Phi(F_1)(w))dw\\
&\le \int_0^{h(\infty)} (1-F_1\circ h^{-1}(w))dw = \int h dF_1<1,
\end{align*}
completing the proof.
\end{proof}

\begin{remark}
When $g$ is continuous, $A=[0, h(\infty)]\cap \R$, and Lemma~\ref{lemma:MFG5} states that $\Phi(F)$ is affine before reaching level $g(1)$. 
\end{remark}

We can now complete the uniqueness argument.

\begin{proof}[Proof of Theorem~\ref{thm:MFE}---Uniqueness.]
Let $F$ be any equilibrium. By Lemma~\ref{lemma:atom}, $\Phi(F)(0)=g(F(0))=g(0+)=g(0)$. Let $\ell_1$ be the strictly increasing affine function given by Lemma~\ref{lemma:MFG5}.
In particular, we have $\ell_1(1)=\Phi(F)^{conc}(1)$ and $\ell_1(\bar w_F)=g(1)$, and $\Phi(F)(w)=\ell_1(w)\wedge g(1)$ whenever 
$F\circ h^{-1}(w)\notin J(g)$ or $w=w_y=\inf\{w\in [0, h(\infty)]\cap\R: F\circ h^{-1}(w)=y\}<\bar w_F$ for some $y\in J(g)$. 

We first find a formula for $\ell_1$.
Observe that either $F\circ h^{-1}(0)=F(0)\notin J(g)$ or $F(0)\in J(g)$ and $w_{F(0)}=0$. In both cases, $\ell_1(0)=\Phi(F)(0)=g(0)<g(1)$. %
We also have $\ell_1(1)=\Phi(F)^{conc}(1)=u^F(x_0)=\bar R$ by symmetry. This completely determines the shape of $\ell_1$; namely, 
\[\ell_1(w)=g(0)+(\bar R-g(0))w.\]

Next, recall $A, A_{y},w_y$ as defined in Lemma~\ref{lemma:MFG5}. We decompose $[0, h(\infty)]\cap \R$ into three disjoint parts: $A_1$, $A\backslash A_1$ and $\bigcup_{y\in J(g)}A_y$. Note that $A_1=[\bar w_F, h(\infty)]\cap\R = [\ell^{-1}_1(g(1)), h(\infty)]\cap\R$ by Lemma~\ref{lemma:MFG4} and Lemma~\ref{lemma:MFG5}\,(i), and that each $A_y$ with $y\in J(g)$ has positive length by Lemma~\ref{lemma:g-jump}.

(i) On $A_1$, we have $F\circ h^{-1} \equiv 1$.

(ii) On $A\backslash A_1$, we have $g(F\circ h^{-1})=\Phi(F)=\ell_1\wedge g(1)=\ell_1$ by Lemma~\ref{lemma:MFG5}\,(ii). The strict monotonicity of $\ell_1$ implies that $F\circ h^{-1}$ is strictly increasing on $A\backslash A_1$. Thus, $A_y$ is a singleton for all $y\notin J(g)\cup\{1\}$. The relation $g(F\circ h^{-1})=\ell_1$ also implies $F\circ h^{-1}\le g^{-1}\circ g (F\circ h^{-1})= g^{-1}(\ell_1)$. In view of Lemma~\ref{lemma:g-flat}, %
the inequality is in fact an equality. Indeed, 
any flat segment of $g$ induces a gap in the range of $F$, which precisely excludes those points $y$ for which $g^{-1}(g(y)) \neq y$, except possibly at the left end point of the flat segment, say $y_1$. The exception only happens if $F$ contains a flat segment at height $y_1$ which is equivalent to $A_{y_1}$ having positive length. Thus, $F\circ h^{-1}=y_1$ is also ruled out on $A\backslash A_1$.

(iii) On each $A_y$ with $y\in J(g)$, we use Lemma~\ref{lemma:MFG5}\,(iii) to obtain $g(y)=\Phi(F)(w_y)=\ell_1(w_y)$, which uniquely determines $w_y$.

In summary, we can decompose $[0, h(\infty)]\cap \R$ into (a)~countably many intervals on which $F\circ h^{-1}$ is flat at some level $y\in J(g)\cup\{1\}$ and (b)~the complementary set $A\backslash A_1$ on which $F\circ h^{-1}=g^{-1}(\ell_1)$.
Each flat segment at level $y\in J(g)$ has left endpoint $w_y=\ell_1^{-1}(g(y))$. To uniquely determine the right-continuous function $F\circ h^{-1}$, it only remains to specify, for each $y\in J(g)$, the right endpoint
\[\tilde w_y:=\sup \{w\in [0, h(\infty)]\cap \R: F\circ h^{-1}(w)=y\}\le \bar w_F\]
of the flat segment. To this end, let $y\in J(g)$ and $y':=\sup\{z\in [0,1]: g(z)=g(y+)\}$. We distinguish two cases:
\begin{itemize}
\item If $y'=y$, then $g$ is non-constant in any right-neighborhood~of $y$, which implies that  $F\circ h^{-1}(\tilde w_y)=y$. (If $F\circ h^{-1}(\tilde w_y)>y$, then $F$ would have an atom at $h^{-1}(\tilde w_y)$ and by Lemma~\ref{lemma:atom}, $g(F\circ h^{-1}(\tilde w_y))=g(y+)$, contradicting the assumption that $y'=y$.)
Let $w^{(m)}>\tilde w_y$ be a sequence such that $w^{(m)}\rightarrow \tilde w_y$ and let $y_m:=F\circ h^{-1}(w^{(m)})$. By right-continuity and the definition of $\tilde w_y$, we have $y_m\rightarrow y$ and $y_m>y$. For large $m$, we may assume $y_m<1$. Observe that $w^{(m)}\ge w_{y_m}>\tilde w_y$, which implies $w_{y_m}\rightarrow \tilde w_y$. 
If $y_m \notin J(g)$, then $w^{(m)} = w_{y_m}\in A\backslash A_1$ and $g(y_m)=\ell_1(w^{(m)})=\ell_1(w_{y_m})$ by (ii) above. If $y_m\in J(g)$, then $g(y_m)=\ell_1(w_{y_n})$ by (iii) above. Combining the two cases and passing to the limit, we obtain $g(y+)=\ell_1(\tilde w_y)$.

\item If $y'>y$, then by Lemma~\ref{lemma:g-flat}, $F$ jumps from $y$ to $y'$ at $F^{-1}(y+)=h^{-1}(\tilde w_y)$. Hence, $F\circ h^{-1}(\tilde w_y)=y'$ and $\Phi(F)(\tilde w_y)=g(y')=g(y+)$. We have either $y'\notin J(g)$ or $y'\in J(g)$ with $\tilde w_y=w_{y'}$, and both lead to $\ell_1(\tilde w_y)=\Phi(F)(\tilde w_y)=g(y+)$.
\end{itemize}
In both cases, we have $\ell_1(\tilde w_y)=g(y+)$, which uniquely determines $\tilde w_y$. 

Putting everything together and taking into account the right-continuity of $F\circ h^{-1}$,
\begin{align*}
F\circ h^{-1}(w)&=\begin{cases}
y, &  \text{if } y\in J(g) \text{ and } w\in [\ell_1^{-1}(g(y)), \ell_1^{-1}(g(y+))),\\
1, & \text{if } w\ge \ell^{-1}_1(g(1)),\\
g^{-1}(\ell_1(w)), & \text{otherwise.}
\end{cases}
\end{align*}
In summary, $F\circ h^{-1}(w)=g^{-1}(\ell_1(w)\wedge g(1))$, or $F(x)=F^*(x)$ after substituting 
$w=h(x)$.
This completes the proof of uniqueness.
\end{proof}

\section{Optimal Reward Design}

Consider a principal who may choose a normalized reward $R$ (i.e., satisfying $R(0)=0$ and $\bar R=1$) and whose goal is to maximize the performance of the top $\alpha\in (0,1)$ fraction of players. More precisely, the aim is to maximize the lowest stopping position of all players in the ranks $[0,\alpha)$,
\begin{align*}
x_{\alpha}
:=F^{-1}((1-\alpha)+) = F_{+}^{-1}(1-\alpha),
\end{align*}
where $F$ is the equilibrium resulting from $R$ and $F^{-1}_{+}$ is the right-continuous inverse of $F$. See Remark~\ref{rk:principal-RC-inverse} below for the technical importance of using $F_{+}^{-1}$, or equivalently, of using the open interval $[0,\alpha)$ when defining the top ranks. Note that the constant 
$\bar\mu_{\infty}=\bar\mu_{\infty}(R)$ in Theorem~\ref{thm:MFE} depends on~$R$. For the following result, we assume $\mu\leq0$ to ensure that $\mu<\bar\mu_{\infty}(R)$ holds for any reward~$R$. Alternately, one may relax this condition to $\mu<\frac{\sigma^{2}}{2x_{0}} \log (\frac{1}{1-\alpha})$ and restrict the principal to rewards~$R$ satisfying $\mu<\bar\mu_{\infty}(R)$.

\begin{theorem}\label{th:design}
  Let $\alpha\in (0,1)$. Then
  \[R^*(r)=\frac{1}{\alpha}1_{[0,\alpha)}(r)\]
  is the unique normalized reward maximizing the performance $x_{\alpha}$. The corresponding  value is $x_{\alpha}^{*}=h^{-1}\left(1/\alpha \right)$ and the equilibrium distribution is
  \[
  F^* = (1-\alpha) 1_{[0,x_\alpha^{*})}+ 1_{[x_{\alpha}^{*},\infty)}.
  \]
\end{theorem} 

\begin{proof}
Let $R$ be an arbitrary normalized reward and $g(y)=R(1-y)$. 
By Theorem~\ref{thm:MFE}, the corresponding mean field equilibrium $F$ is unique and
\[F(x)=g^{-1}\left(h(x) \wedge g(1)\right)= 1\wedge \inf\{y: g(y)>h(x)\wedge g(1)\}.\]
We have $F(x)>1-\alpha$ if and only if $g(1-\alpha+\eps)\le h(x)\wedge g(1)$ for some $\eps=\eps(x)>0$, hence
\begin{align*}
F_+^{-1}(1-\alpha)
&=\inf\{x\ge 0: g(1-\alpha+\eps)\le h(x)\wedge g(1) \text{ for some $\eps>0$}\} \\
&=h^{-1}(g((1-\alpha)+)).
\end{align*}
As $h^{-1}$ is strictly increasing, maximizing this quantity is equivalent to maximizing $g((1-\alpha)+)$. Recalling that $g$ is monotone, left-continuous and $\int_0^1 g(y)dy=1$, the unique maximizer is given by $g^*(y):=R^{*}(1-y)$ and the corresponding maximum value is $h^{-1}\left(1/\alpha \right)$.
By Theorem~\ref{thm:MFE} (or Proposition~\ref{prop:MFG-cutoff}), the corresponding equilibrium is~$F^*$.
\end{proof} 

Comparing with the results cited in the Introduction (and recalled in more detail in Section~\ref{se:convOfOptDesign} below), Theorem~\ref{th:design} gives a clear-cut answer to a question which remained partially open in the $n$-player setting. On the other hand, the result illustrates that the mean field analysis alone could easily lead to an oversimplified picture: the optimal design in the $n$-player game is \emph{not} given by the cut-off reward at the target rank, in most cases. See also Section~\ref{se:MFconvergence} for further comparison of mean field and $n$-player games.

\begin{remark}\label{rk:principal-RC-inverse}
The principal's goal is to maximize $F_{+}^{-1}(1-\alpha)$ rather than the quantile $F^{-1}(1-\alpha)$. Indeed, the worst performance among the top $\alpha$-fraction of players need not be the same as the best performance among the bottom $(1-\alpha)$-fraction. The equilibrium $F^{*}$ of Theorem~\ref{th:design} has an atom of size $\alpha$ at $x^{*}_{\alpha}$ and an atom of size $1-\alpha$ at the origin. Thus, $(F^{*})_{+}^{-1}(1-\alpha)=x^{*}_{\alpha}$ but $(F^{*})^{-1}(1-\alpha)=0$.

It is crucial to formulate the principal's problem in the form stated above: if instead we aim to maximize the best performance in the quantile $F^{-1}(1-\alpha)$, the optimization \emph{fails to admit a solution.} 
To see this, note that for each $m\geq1$, the cutoff reward 
$R^{(m)}(r):=\frac{1}{\alpha+1/m}1_{[0,\alpha+1/m)}(r)$
gives rise to the equilibrium $F^{(m)}=(1-\alpha-1/m) 1_{[0,x_m)}+ 1_{[x_m,\infty)}$ where $x_m=h^{-1}(1/(\alpha+1/m))$. Moreover, $(F^{(m)})^{-1}(1-\alpha)=x_{m}$  increases to $h^{-1}(1/\alpha)$ as $m\to\infty$. However, there exists no equilibrium distribution $F$ achieving $F^{-1}(1-\alpha)=h^{-1}(1/\alpha)$. Indeed, by Theorem~\ref{th:design}, such $F$ would have to coincide with~$F^{*}$, but $(F^{*})^{-1}(1-\alpha)=0<h^{-1}(1/\alpha)$.
\end{remark}

\begin{remark}\label{rk:PoA}
In analogy to the ``price of anarchy'' we can compare the principal's optimization over equilibria with a different problem where the planner can dictate the players' stopping strategy (regardless of equilibrium considerations). This problem can be stated as
\[\max_{F\in \cF} F^{-1}_{+}(1-\alpha).\]
Using Lemma~\ref{le:attainableDistrib} we can check that the unique solution is $F=F^{*}$, the equilibrium distribution of Theorem~\ref{th:design}. In particular, the ``welfare'' of the second-best principal who can only choose the reward function is equal to that of a planner who can dictate strategies.

This consideration also shows a different avenue to  Theorem~\ref{th:design}: if one is only interested in this specific question rather than mean field equilibria for general reward functions, one can first argue that $\argmax_{F\in \cF} F^{-1}_{+}(1-\alpha)=F^*$ and then,  as in Proposition~\ref{prop:MFG-cutoff}, that $F^{*}$ is the unique equilibrium for $R^{*}$.

\end{remark}

\section{Convergence to the Mean Field}\label{se:MFconvergence}

To formulate the $n$-player game associated with our mean field contest, fix a decreasing, non-constant reward vector $(R_{1},\dots,R_{n})$. Here $R_{1}$ is interpreted as the reward for the best rank whereas $R_{n}$ is the worst. As in the mean field game, the players are ranked according to their level of stopping and ties are split uniformly at random. The set $\cF$ of feasible stopping distributions remains the same and the definition of equilibrium is analogous. It is shown in~\cite{NutzZhang.21a} that the $n$-player game admits a unique equilibrium $F^*_n\in\cF$ as soon as the drift $\mu$ satisfies
\[
\mu< \bar \mu_n:=\frac{\sigma^{2}}{2x_{0}} \log\left( \frac{R_{1} - R_{n}}{R_{1} -\bar R_{n}}\right), \quad\mbox{where}\quad \bar R_{n}:=\frac{1}{n}\sum_{k=1}^n R_k.
\]
The equilibrium distribution has compact support $[0,\bar{x}_{n}]$ and cdf
\[F^*_n(x)=g_n^{-1}(u^*_n(x)),
\quad\mbox{where}\quad \bar x_n=h^{-1}\left(\frac{R_1-R_n}{\bar R-R_n}\right) \quad\mbox{and}
\quad \]
\[g_n(y)=\sum_{k=1}^{n} R_{k} {{n-1}\choose{k-1}} y^{n-k}(1-y)^{k-1}, \quad\mbox{and}
\quad u^*_n(x)
=\left[R_n+\left(\bar R_{n}-R_n\right)h(x)\right]\wedge R_1\]
is the equilibrium value function.
In contrast to the mean field setting, $F^*_n$ is always atomless. %
Moreover, $g_{n}$ is strictly increasing and smooth, hence so is its (true) inverse $g_{n}^{-1}$.

\subsection{Convergence of the $n$-Player Equilibrium}

The next result shows that if the reward vector is induced by a reward function for the mean field game, the $n$-player equilibrium distributions and  value functions converge to their mean field counterparts as described in Theorem~\ref{thm:MFE}. 

\begin{theorem}\label{th:equilibConvergence}
Let $R:[0,1]\rightarrow \R_+$ be a reward function, $\mu<\bar \mu_\infty$, and define $R_k:=R(k/n)$ for $k=1, \ldots, n$. Then as $n\to\infty$, using $(R_{1},\dots,R_{n})$ as reward for the $n$-player game and~$R$ for the mean field game, the associated unique equilibrium distributions converge weakly and the equilibrium value functions converge uniformly on compact sets.
\end{theorem}

\begin{remark}
  If we consider a generalized reward function~$R$ with $R(1-)<R(1)$ as discussed in Remark~\ref{rk:LC-necessary}, the limit of the $n$-player equilibria selects a particular equilibrium among the infinitely many mean field equilibria; namely, the one detailed in Theorem~\ref{thm:MFE}. This follows from the fact that the proof of Theorem~\ref{th:equilibConvergence} does not use the condition $R(1-)=R(1)$.
\end{remark}

Before proceeding with the proof, we state a formula that will be used in later arguments as well. Consider the empirical cdf of i.i.d.\ uniform random variables $\{U_i\}_{i=1, \ldots, n-1}$ on $[0,1]$, 
\[\hat F_{n-1}(y)=\frac{1}{n-1}\#\{i: U_i\le y\}.\]
Let $0\le y_1\le y_2\le 1$. 
Among the $n-1$ random variables~$\{U_i\}$, there are $I_{n-1}=(n-1)(1-\hat F_{n-1}(y_2))$ with values above $y_2$, $J_{n-1}=(n-1)\hat F_{n-1}(y_1)$ below $y_1$, and $K_{n-1}=(n-1)(\hat F_{n-1}(y_2)-\hat F_{n-1}(y_1))$ in-between $y_1$ and $y_2$. 
Thus, we have the following formula for any function $\phi(i,j,k)$: 
\begin{equation}\label{eq:payoff}
\begin{aligned}
&\sum_{\substack{i,j,k\ge 0 \\ i+j+k= n-1}} \phi(i,j,k) {{n-1}\choose{i,j,k}}(1-y_2)^i y_1^j (y_2-y_1)^k\\
&=\sum_{\substack{i,j,k\ge 0 \\ i+j+k= n-1}} \phi(i,j,k) P(I_{n-1}=i, J_{n-1}=j, K_{n-1}=k)\\
&=E\left[\phi\left((n-1)(1-\hat F_{n-1}(y_2)), (n-1)\hat F_{n-1}(y_1), (n-1)(\hat F_{n-1}(y_2)-\hat F_{n-1}(y_1))\right)\right].
\end{aligned}
\end{equation}

\begin{proof}[Proof of Theorem~\ref{th:equilibConvergence}.]
We have $R_n=R(1)$ and $R_1=R(1/n)\rightarrow R(0)$ by the right-continuity of~$R$. Moreover, the Riemann sum $\frac{1}{n}\sum_{k=1}^n R_k\rightarrow \int_0^1 R(r)dr=\bar R$. It follows that $\bar \mu_n\rightarrow \bar \mu_\infty$, so that $\mu<\bar \mu_\infty$ ensures $\mu<\bar \mu_n$ for all $n$ sufficiently large and the equilibria are uniquely defined. The pointwise convergence of $u^*_n$ to $u^*$ is clear from their respective formulas. As these functions are increasing and $u^*$ is continuous, the pointwise convergence is locally uniform (see e.g.\ \cite[Proposition 2.1]{Resnick.07}).

To show the weak convergence of the equilibrium distributions, we prove $F_n^\ast(x)\rightarrow F^*(x)$ whenever~$x$ is a point of continuity of $F^*$.
We first argue that
\begin{equation}\label{eq:gn-conv}
g_n\rightarrow g \quad \text{at  at every point of continuity of $g$.}
\end{equation}
Taking $y_1=y_2=y$ and $\phi(i,j,k)=R_{i+1}=R((i+1)/n)$ in \eqref{eq:payoff}, we obtain
\begin{equation}\label{eq:gn}
g_n(y)%
=E \left[R\left(\frac{(n-1)(1-\hat F_{n-1}(y))+1}{n}\right)\right]=E\left[g\left(\frac{n-1}{n}\hat F_{n-1}(y)\right)\right].
\end{equation}
By the strong law of large numbers, $\hat F_{n-1}(y)\rightarrow y$ a.s.\ for each $y$. If $y$ is a point of continuity of $g$, it follows that $g\big(\frac{n-1}{n}\hat F_{n-1}(y) \big)\rightarrow g(y)$ a.s.\ and the bounded convergence theorem yields $g_n(y)\rightarrow g(y)$ as claimed.

We have $F_{n}^*(x)=g_{n}^{-1}(z_{n})$ for $z_{n}:=u^*_n(x)$ and similarly $F^*(x)=g^{-1}(z)$ for $z:=u^*(x)$. By the above, $z_{n}\to z$  for all $x$, and therefore we need to show that $g^{-1}_{n}(z_{n})\to g^{-1}(z)$ whenever $z\in C$, where $C$ is the set of continuity points of $g^{-1}$.
Up to normalization, we may think of $g_{n},g$ as cdf of weakly converging distributions. It is then known that the inverses $g_{n}^{-1}$ converge to $g^{-1}$ on the set where the left- and right-continuous inverses of $g$ coincide, and hence on~$C$ (cf.\ the proof of 	\cite[Theorem 3.2.2, p.\,100]{Durrett.10}).
We have $g_{n}^{-1}(z\pm\eps)\to g^{-1}(z\pm\eps)$ for $z\in C$ and $\eps>0$ with $z\pm\eps\in C$. Using that $g_{n}^{-1},g^{-1}$ are monotone, we deduce that $g_{n}^{-1}(z_{n})\to g^{-1}(z)$ whenever $z_{n}\to z$ and $z\in C$, as desired.
\end{proof}

\begin{remark}
A discussion related to Theorem~\ref{th:equilibConvergence} can be found in the work \cite{FangNoe.16} on $n$-player capacity-constrained contests, which can be related to the present game via Skorokhod embedding. (Some results of the preprint~\cite{FangNoe.16} were later published as~\cite{FangNoeStrack.20}.) Namely, \cite[Proposition~9]{FangNoe.16} studies the effect of scaling the $n$-player contest by multiplying the number of participants while dividing the reward at each rank.  
This basically corresponds to taking $\lim_{n} F_{n}^{*}$, if only for the particular case where~$R$ is a step function. An infinite player game is not considered, so that the limiting distribution~$F^{*}$ cannot be recognized as a mean field object. Instead, the authors derive an involved algorithm~\cite[Remark~A-1]{FangNoe.16} to construct~$F^{*}=\lim_{n} F_{n}^{*}$. The limit is a step function in this particular case, and the algorithm determines the~$n-1$ jump locations and magnitudes. It seems that the simple representation~\eqref{eq:MFE}, or the game-theoretic meaning of~$F^{*}$, were not identified.
\end{remark}

\subsection{$\eps$-Nash Equilibrium Property of the Mean Field Strategy}

Recall that $\xi^F(x)$ denotes the payoff for stopping at~$x$ if all other players in the mean field game use~$F$; cf.~\eqref{eq:defxi}. Analogously, we can define the expected payoff $\xi^F_n(x)$ in the $n$-player game. 
We say that
$F^*\in\cF$ is an \emph{$\eps$-Nash equilibrium} of the $n$-player game if 
\[\int \xi^{F^*}_n dF^*\ge \int \xi^{F^*}_n dF-\eps \quad \mbox{for all}\quad F\in\cF.\]
That is, a player deviating unilaterally from~$F^*$ can improve her expected payoff by at most~$\eps$. Correspondingly, $F^*$ is an \emph{$o(1)$-Nash equilibrium} as $n\to \infty$ if for any $\eps>0$,  the above holds for all large~$n$, or equivalently
\[\lim_{n\rightarrow\infty }\sup_{F\in\cF}\left(\int \xi_n^{F^*}dF - \int \xi_n^{F^*}dF^*\right)=0.\]
We can now state the main result of this subsection.

\begin{theorem}\label{th:epsEquilibrium}
Let $R$ be a reward function, $\mu<\bar \mu_\infty$, and let $F^*$ be the associated mean field equilibrium. Define $R_k:=R(k/n)$, $k=1, \ldots, n$ as reward for the $n$-player game. Then $F^*$ is an $o(1)$-Nash equilibrium of the $n$-player game as $n\to \infty$ if and only if~$R$ is continuous.
\end{theorem}

The positive result in Theorem~\ref{th:epsEquilibrium} is consistent with a large body of literature; cf.\ the Introduction. That the continuity condition is sharp, may be surprising. Indeed we will show that if~$R$ has a jump and $\eps>0$ is small enough, then $F^{*}$ is not an $\eps$-Nash equilibrium, for \emph{all} large~$n$. This is not related to atoms in the equilibrium but rather to the gap in the support of~$F^{*}$ caused by the jump $R(x)-R(x-)$ in reward and a stochastic knife-edge phenomenon. The idea of the proof is that a player can improve by suitably shifting some mass of the stopping distribution \emph{into the gap.} A level of stopping inside the gap would imply the reward~$R(x)$ in the mean field game, but in the $n$-player game, the result depends on the sample---the reward is approximately~$R(x)$ in roughly half the samples, but the higher reward~$R(x-)$ in the other half. By shifting more mass from below the gap than from above (all while maintaining feasibility), the player can increase the payoff relative to~$F^{*}$.

The proof of Theorem~\ref{th:epsEquilibrium} occupies the remainder of this subsection. Throughout the proof, the rewards and~$F^{*}$ are defined as in Theorem~\ref{th:epsEquilibrium}. As a first step, we derive a convenient formula for $\xi^F_n(x)$. The probability that among players $2, \ldots, n$, there are exactly $i$ players stopping above $x$, $j$ players below $x$, and $k$ players at $x$, is given by
\[ {{n-1}\choose{i,j,k}}(1-F(x))^i F(x-)^j (F(x)-F(x-))^k.\]
Such a configuration leads to an average payoff 
$(R_{i+1}+\cdots+ R_{i+k+1})/(k+1)$
for player~$1$ as ties are broken randomly. It follows that
\begin{align*}
\xi^F_n(x)&=\sum_{\substack{i,j,k\ge 0 \\ i+j+k= n-1}} \frac{R_{i+1}+\cdots+ R_{n-j}}{k+1} {{n-1}\choose{i,j,k}}(1-F(x))^i F(x-)^j (F(x)-F(x-))^k.
\end{align*}
This reduces to $g_n(F(x))$ if $F(x)=F(x-)$. Taking $\phi(i,j,k)=(R_{i+1}+\cdots+ R_{i+k+1})/(k+1)$ in \eqref{eq:payoff}, we have the alternative representation
\begin{equation}\label{eq:xi_n}
\xi^F_n(x)
=E\left[\frac{\sum_{\ell=(n-1)(1-\hat F_{n-1}(F(x)))+1}^{n-(n-1)\hat F_{n-1}(F(x-))} R_\ell}{(n-1)(\hat F_{n-1}(F(x))-\hat F_{n-1}(F(x-)))+1}\right].
\end{equation}

\begin{lemma}\label{lemma:xi-conv0}
Let $F\in \cF$ have an atom at~$x$. Then 
$\lim_n \xi_n^F(x)=\xi^F(x)$.
\end{lemma}
\begin{proof}
Let $F$ have an atom at $x$. Write $y_1=F(x-)$ and $y_2=F(x)$.
By \eqref{eq:xi_n},
\begin{align*}
\xi^F_n(x)
&=E\left[\frac{\sum_{\ell=n-(n-1)\hat F_{n-1}(y_2)}^{n-(n-1)\hat F_{n-1}(y_1)} g\left(\frac{n-\ell}{n}\right)}{(n-1)(\hat F_{n-1}(y_2)-\hat F_{n-1}(y_1))+1}\right]\\
&=E\left[\frac{n}{(n-1)(\hat F_{n-1}(y_2)-\hat F_{n-1}(y_1))+1}\sum_{\ell=(n-1)\hat F_{n-1}(y_1)}^{(n-1)\hat F_{n-1}(y_2)} g\left(\frac{\ell}{n}\right)\frac{1}{n}\right].
\end{align*}
Using the a.s.\ convergence of $\hat F_{n-1}(y)$ to $y$, we deduce that
\begin{equation*}
\xi^F_n(x)\rightarrow \frac{1}{y_2-y_1}\int_{y_1}^{y_2}g(y)dy =\xi^F(x).
\tag*{\qedhere}
\end{equation*}
\end{proof}

\begin{lemma}\label{lemma:xi-cov}
Let $R$ be continuous. Then $\xi^{F^*}_n(\cdot)$ converges to $\xi^{F^*}(\cdot)$ uniformly.
\end{lemma}

\begin{proof}
We first show that $\xi^{F^*}_n$ converges to $\xi^{F^*}$ pointwise. The convergence at points of discontinuity of~$F^*$ holds by Lemma~\ref{lemma:xi-conv0}. At points of continuity, we have $\xi_n^{F^*}=g_n \circ F^*$ and $\xi^{F^*}=g \circ F^*$. The pointwise convergence then follows from \eqref{eq:gn-conv} and the assumed continuity of $g$.

By Theorem~\ref{thm:MFE}, $F^*$ has compact support $[0, \bar x]$. For $x>\bar x$ it is clear that $|\xi_n^{F^*}(x)-\xi^{F^*}(x)|=|g_n(1)-g(1)|=|R(1/n)-R(0)|\rightarrow 0$. To see that the convergence is also uniform on $[0, \bar x]$, we note that $\xi_n^{F^*}$ is an increasing function for each~$n$. Moreover, the pointwise limit $\xi^{F^*}$ is continuous: as $g$ is continuous, we have $g(g^{-1}(z))=z$ and then $\xi^{F^*}(x)=g(F^*(x))=[R(1)+(\bar R-R(1))h(x)]\wedge R(0)$ is continuous as well; cf.\ Lemma~\ref{lemma:atom} and Theorem~\ref{thm:MFE}. A standard argument for monotone functions
then yields that the pointwise convergence is uniform.
\end{proof}

\begin{lemma}\label{lemma:CLT}
If $g$ has a jump at $y\in (0,1)$, then $\liminf_n g_n(y) \ge [g(y)+g(y+)]/2$.\footnote{The reverse inequality $\limsup_n g_n(y) \le [g(y)+g(y+)]/2$ also holds, but is not needed for our purposes.}
\end{lemma}
\begin{proof}
Let $y'<y$. Using \eqref{eq:gn}, we have 
\begin{align*}
g_n(y)&=E\left[g\left(\frac{n-1}{n}\hat F_{n-1}(y)\right)\right]\\
&\ge g(y') P\left(y'<\frac{n-1}{n}\hat F_{n-1}(y)\le y\right) + g(y+)P\left(\frac{n-1}{n}\hat F_{n-1}(y)> y\right)\\
&=g(y')P\left(\frac{n-1}{n}\hat F_{n-1}(y)> y'\right)+\left(g(y+)-g(y')\right)P\left(\frac{n-1}{n}\hat F_{n-1}(y)> y\right).
\end{align*}
The strong law of large numbers implies $\hat F_{n-1}(y)\rightarrow y$ a.s.\ and hence
\[P\left(\frac{n-1}{n}\hat F_{n-1}(y)> y'\right)\rightarrow 1.\]
By the central limit theorem, $\sqrt{n-1}\left(\hat F_{n-1}(y)-y\right)/\sqrt{y(1-y)}$ converges to $\cN(0,1)$ in distribution. It follows that for any fixed $\gamma>0$ and $n\ge 1+y/[(1-y)\gamma^2]$,
\begin{align*}
P\left(\frac{n-1}{n}\hat F_{n-1}(y)> y\right)&=P\left(\frac{\sqrt{n-1}\left(\hat F_{n-1}(y)-y\right)}{\sqrt{y(1-y)}}> \frac{1}{\sqrt{n-1}}\sqrt{\frac{y}{1-y}}\right)\\
&\ge P\left(\frac{\sqrt{n-1}\left(\hat F_{n-1}(y)-y\right)}{\sqrt{y(1-y)}}> \gamma \right)\rightarrow 1-N(\gamma),
\end{align*}
where $N(\cdot)$ is the standard normal cdf. Combining the two limits, we obtain
\[\liminf_n g_n(y) \ge g(y')+(g(y+)-g(y'))(1-N(\gamma))=g(y')N(\gamma)+g(y+)(1-N(\gamma)).\]
In view of the left-continuity of $g$, sending $\gamma\rightarrow 0$ and $y'\rightarrow y$ concludes the proof. 
\end{proof}

\begin{proof}[Proof of Theorem~\ref{th:epsEquilibrium}.]
\emph{Part~1: Sufficiency.} Let $R$ be continuous and $\eps>0$.  Lemma~\ref{lemma:xi-cov} shows the existence of $n_\eps$ such that $\|\xi^{F^*}_n-\xi^{F^*}\|_\infty<\eps/2$ whenever $n\ge n_\eps$. Let $F\in \cF$. For $n\ge n_\eps$, noting that $\int \xi^{F^*} dF- \int \xi^{F^*} dF^*\leq0$ by the equilibrium property of~$F^*$,
\begin{align*}
&\int \xi^{F^*}_n dF - \int \xi^{F^*}_n dF^* \\
 &=\int \xi^{F^*}_n dF - \int \xi^{F^*} dF + \int \xi^{F^*} dF- \int \xi^{F^*} dF^* + \int \xi^{F^*} dF^* -\int \xi^{F^*}_n dF^*\\
& \le \int \left|\xi^{F^*}_n-\xi^{F^*}\right| dF + \int \left|\xi^{F^*}_n-\xi^{F^*}\right| dF^*<\eps.
\end{align*}
This proves the $o(1)$-Nash property of $F^*$.

\emph{Part~2: Necessity.} Let $g(y)=R(1-y)$ have a jump at $y_0\in (0,1)$. We show the stronger statement
\begin{equation}\label{claim:epsNE}
\sup_{F\in \cF}\liminf_n \left(\int \xi_n^{F^*}dF - \int \xi_n^{F^*}dF^* \right)>0.
\end{equation}
Let 
\[a=h^{-1}\left(\frac{g(y_0)-R(1)}{\bar R-R(1)}\right)\quad \text{and} \quad b=h^{-1}\left(\frac{g(y_0+)-R(1)}{\bar R-R(1)}\right).\]
By Theorem~\ref{thm:MFE}, the associated mean field equilibrium $F^*$ is flat on $[a,b)$ and $F^*(a-\eta)<F^*(a)=y_0<F^*(b+\eta)$ for any $\eta>0$. Suppose players~$2, \ldots, n$ all use~$F^*$ with associated measure~$\nu^*$ and player~1 considers an alternative strategy of the form
\[\nu=\nu^*- \zeta + |\zeta|\delta_{a'} \]
for some $a' \in (a,b)$ and a subprobability $\zeta\le \nu^*$ with density $d\zeta/d\nu^*= f$. To ensure the feasibility of $\nu$, we require $\int h d\nu = \int h d\nu^*$ which translates to
\begin{equation}\label{eg-eq0}
 \int \left(h(a')-h\right)f d\nu^*=0.
\end{equation}
Our goal is to obtain a lower bound for the payoff difference
\begin{equation}\label{eg-eq1}
\int \xi_n^{F^*}d\nu - \int \xi_n^{F^*}d\nu^*=|\zeta| \xi_n^{F^*}(a')-\int  \xi_n^{F^*} f  d\nu^*
\end{equation}
 that is independent of $n$ for $n$ large.
Since $F^*$ is continuous at $a'$, we have $\xi^{F^*}_n(a')=g_n(F^*(a'))=g_n(y_0)$. By Lemma~\ref{lemma:CLT}, 

\begin{equation}\label{eg-eq2}
\liminf_n \xi^{F^*}_n(a') =\liminf_n g_n(y_0) \ge \frac{g(y_0)+g(y_0+)}{2}.
\end{equation}
Let $x\ge 0$. If $F^*$ is continuous at $x$, we use \eqref{eq:gn-conv} to get 
$\limsup_n\xi_n^{F^*}(x)= \limsup_n g_n(F^*(x))\le g(F^*(x)+)$, whereas
if $F^*$ has a jump at $x$, we use Lemma~\ref{lemma:xi-conv0} to get 
$\xi_n^{F^*}(x)\rightarrow \xi^{F^*}(x)=g(F^*(x)).$
Reverse Fatou's lemma then implies
\begin{align}\label{eg-eq3}
\limsup_n \int \xi_n^{F^*}f d\nu^* & \le \int  g(F^*(\cdot)+)f d\nu^*|_{\R_+\backslash\{a\}}+g(F^*(a))f(a)\nu^*\{a\}.
\end{align}
Substituting \eqref{eg-eq2} and \eqref{eg-eq3} into \eqref{eg-eq1}, we obtain
\begin{align*}
&\liminf_n \left(\int \xi_n^{F^*}d\nu - \int \xi_n^{F^*}d\nu^*\right) \\
&\ge \frac{g(y_0)+g(y_0+)}{2} \int f d\nu^*-\int g(F^*(\cdot)+)f d\nu^*|_{\R_+\backslash\{a\}}-g(y_0)f(a)\nu^*\{a\} \notag\\
&= \frac{g(y_0+)-g(y_0)}{2}f(a)\nu^*\{a\}+ \int \left(\frac{g(y_0)+g(y_0+)}{2}-g(F^*(\cdot)+)\right)f d\nu^*|_{[0,a)} \\
&\quad- \int \left(g(F^*(\cdot)+)-\frac{g(y_0)+g(y_0+)}{2}\right)f d\nu^*|_{[b,\infty)}.
\end{align*}
As $F^*(x)<y_0$ for $x<a$ and thus $g(F^*(x)+)\le g(y_0)=a$, we can further bound the above expression from below by
\begin{equation*}\label{eg-eq4}
C_f := \frac{g(y_0+)-g(y_0)}{2}\int f d\nu^*|_{[0,a]}- \int \left(R(0)-\frac{g(y_0)+g(y_0+)}{2}\right)f d\nu^*|_{[b,\infty)}.
\end{equation*}
It remains to show that by choosing a suitable Radon--Nikodym derivative $f$, the lower bound $C_f$ for the expected improvement can be made strictly positive. 
To this end, we pick
\[f(x)=1_{(a-\eta, a]}(x)+\lambda 1_{[b,\infty)}(x)\]
for some constants $\lambda \in [0,1]$ and $\eta> 0$ to be determined. With this form of $f$, we always have $0\le \zeta \le \nu^*$ and the feasibility condition \eqref{eg-eq0} becomes
\[\lambda =\frac{\int \left(h(a')-h\right) d\nu^*|_{(a-\eta,a]}}{\int \left(h-h(a')\right) d\nu^*|_{[b,\infty)}}\in \left(0, \frac{h(a')-h((a-\eta)\vee 0)}{h(b)-h(a')} \cdot \frac{\nu^*(a-\eta, a]}{\nu^*[b,\infty)}\right].\]
We use this equality as definition for~$\lambda$. Then
\begin{align*}
C_f&=\frac{g(y_0+)-g(y_0)}{2}\nu^*(a-\eta,a] -\lambda \left(R(0)-\frac{g(y_0)+g(y_0+)}{2}\right) \nu^*[b,\infty)\\
&\ge \nu^*(a-\eta, a]\left(\frac{g(y_0+)-g(y_0)}{2}-\left(R(0)-\frac{g(y_0)+g(y_0+)}{2}\right)\frac{h(a')-h((a-\eta)\vee 0)}{h(b)-h(a')}\right),
\end{align*}
where the inequality is derived by replacing $\lambda$ by its upper bound.
Choose $a'- a$ and $\eta$ sufficiently small so that
\[\frac{h(a')-h((a-\eta)\vee 0)}{h(b)-h(a')}< \min\left(\nu^*[b,\infty), \frac{g(y_0+)-g(y_0)}{2R(0)-g(y_0)-g(y_0+)}\right).\]
Then $\lambda\in (0,1)$ and $C_f>0$. This concludes the proof of \eqref{claim:epsNE} and hence of the theorem.
\end{proof}

\subsection{Convergence of the Optimal Reward Design}\label{se:convOfOptDesign}

We have seen in Theorem~\ref{th:design} that the optimal design to maximize performance at a given target rank~$\alpha$ is the cut-off reward at that same rank. As mentioned in the Introduction, the best design in the prelimit is more complicated: for the $n$-player game with zero drift, the cut-off at a certain rank $k^{*}_{n}$ is optimal for the expected performance at target rank~$k$. A formula (recalled below) for~$k^{*}_{n}$ was found in~\cite{NutzZhang.21a}, and it is also noted that $k^{*}_{n}\geq k$, with $k^{*}_{n}> k$ unless $k$ or $k/n$ are small. For drift $\mu>0$, a cut-off is again optimal, but the exact location of the cut-off is not known, whereas for $\mu<0$, the optimal shape can look smoother than the sharp cut-off. In this section, we numerically compare the $n$-player game with the mean field limit for large~$n$, focusing on $\mu=0$ in order to have an exact result available for finite~$n$.

We recall from \cite[Proposition~3.11]{NutzZhang.21a} that the optimal normalized reward for the expected $k$-th rank performance in the $n$-player game with $\mu=0$ is the cut-off at~$k^{*}_{n}$ (i.e., $R_i = 1/k^{*}_{n}$ for $i\le k^{*}_{n}$ and $R_i =0$ for $i>k^{*}_{n}$), where~$k^{*}_{n}$ is determined as
\[k^{*}_{n}=\max\left\{j\ge k: \phi(k,j)\ge \frac{1}{j-1} \sum_{l=1}^{j-1}\phi(k,l)\right\}, \quad \phi(k,l):=\frac{(2n-k-l)!(k+l-2)!}{(n-l)!(l-1)!}.\]
The corresponding expected $k$-th rank performance is  
 \begin{equation}\label{eq:k-n-RankPerformance}
    nx_{0} \frac{n!}{(2n-1)!} {{n-1}\choose{k-1}} \frac{1}{k^{*}_{n}}  \sum_{l=1}^{k^{*}_{n}}\phi(k,l).
  \end{equation}
If we scale $k$ proportionally to $n$ by fixing $k/n \approx \alpha \in (0,1)$, we can compare the optimal cut-off ratio~$k^{*}_{n}/n$ with the mean field optimal cut-off~$\alpha$. In the numerical example, we consider the median performance; i.e., $\alpha =0.5$. A similar behavior can be observed for other choices of~$\alpha$.

\begin{figure}[t]
\centering
\includegraphics[width=\textwidth]{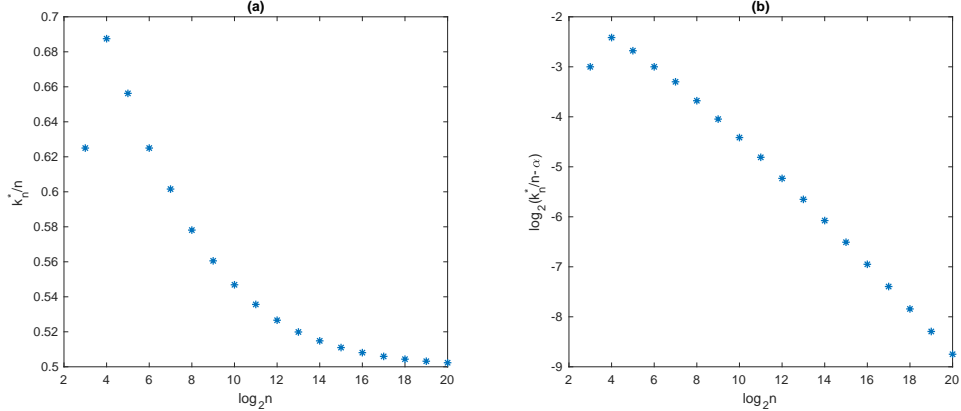}
\caption{(a) Convergence of the optimal cut-off ratio $k^{*}_{n}/n$ to $\alpha$. (b) Log-log plot of the difference $k^{*}_{n}/n-\alpha$, illustrating that $k^{*}_{n}/n$ converges to $\alpha$ at rate approximately $O(n^{-r})$ for a fractional power~$r$. Increments of~$n$ in all plots are chosen such as to avoid rounding effects related to the fact that $k^{*}_{n}$ must be integer. On a finer scale for~$n$, there are oscillations (cf.\ \cite[Figure~3]{NutzZhang.21a}) which however disappear in the large $n$ limit.}
\label{fig:conv_cutoff}
\end{figure}

Figure~\ref{fig:conv_cutoff} shows that $k^{*}_{n}/n$ converges
to $\alpha$ as $n\rightarrow\infty$. The convergence is rather slow; e.g., for $n=1024$, the optimal cut-off rank is still more than 9\% larger than the mean-field optimum. 
This already suggests that using the mean field optimal design as a proxy for the $n$-player design may be problematic at least for moderate~$n$. %

Next, we consider the quality of the mean field proxy from the point of view of the principal: we fix the optimal design~$R^{*}$ from the mean field setting (Theorem~\ref{th:design}) and compare the resulting expected performance in the $n$-player game with the performance~\eqref{eq:k-n-RankPerformance} of the exact optimizer given by $k^{*}_{n}$. For comparison, we mention that the analogous question was considered in the Poissonian model of~\cite{NutzZhang.19}, for the same performance functional of the principal, and there the mean field proxy was shown to be $O(1/n)$-optimal for the $n$-player design problem.

\begin{figure}[t]
\centering
\includegraphics[width=\textwidth]{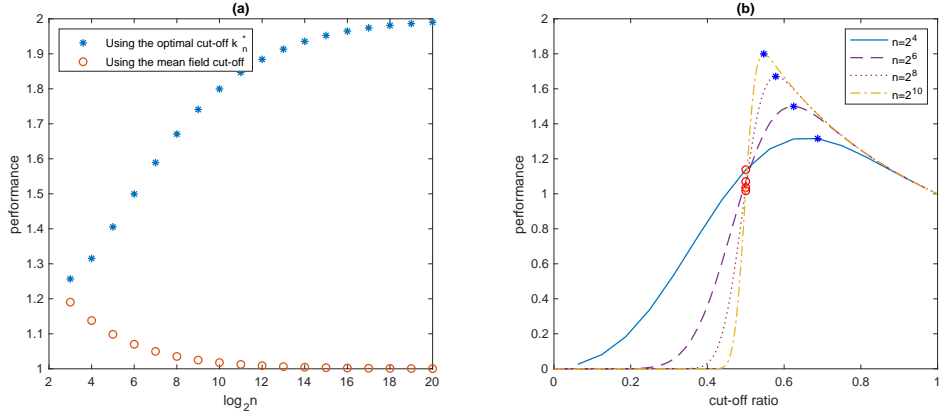}
\caption{Red circles correspond to the mean field proxy (cut-off at rank $\alpha n$ or ratio $\alpha$), blue stars correspond to the exact $n$-player optimizer (cut-off at rank $k^{*}_{n}$ or ratio $k^{*}_{n}/n$). Here $\alpha=0.5$ and $x_{0}=1$. (a)~Performance of the mean field proxy diverges from the optimal design given by~$k^{*}_{n}$. (b)~Median player's performance for all cut-off schemes when $n=2^4,\,2^6,\,2^8,\,2^{10}$.  As $n$ increases, the blue and red points converge in the horizontal direction but nevertheless diverge in the vertical direction.}
\label{fig:div}
\end{figure}

Figure~\ref{fig:div}\,(a) shows not only that the performance of the proxy may be significantly inferior for finite~$n$, but indeed that the performances diverge as $n\to\infty$, with the exact solution performing twice as well. The performance of the exact solution converges to the optimal performance in the mean field model as stated in Theorem~\ref{th:design}, $x^*_\alpha=h^{-1}(1/\alpha)=2x_0$, but the performance of the proxy does not.

Figure~\ref{fig:div}\,(b) plots the same data points for some values of~$n$, together with curves showing the performance of any cut-off strategy as a function of the cut-off location. For larger~$n$, the curves are increasingly steep in a left neighborhood of the maximum: the vertical distance between the data points increases even though the horizontal distance decreases. In other words, the performance of $R^{*}$ is increasingly inferior despite the cut-off location approximating the optimal location. 

The reason lies in the lack of smoothness of the mean field game. Indeed, we know that the equilibrium distribution~$F^{*}_{n}$ induced by $R^{*}$ in the $n$-player game converges weakly to the mean field equilibrium~$F^{*}$ which is a two-point distribution  (Theorems~\ref{th:equilibConvergence} and~\ref{th:design}). While~$F^{*}_{n}$ is increasingly concentrated around the location of the limiting atoms at~$0$ and~$x^{*}_{\alpha}$ for large~$n$, the distribution is still smooth with connected support for finite~$n$, so that the $(1-\alpha)$-quantile stretches far beyond~$x^{*}_{\alpha}$, causing the inferior performance.

We emphasize that the reason for the poor quality of the proxy observed here is very different from the knife-edge phenomenon leading to the negative result in Theorem~\ref{th:epsEquilibrium}, and quite possibly more relevant to applications.

\bibliography{stochfin}

\newcommand{\dummy}[1]{}
\begin{thebibliography}{10}

\bibitem{AnkirchnerKaziTaniWendtZhou.21}
S.~Ankirchner, N.~Kazi-Tani, J.~Wendt, and C.~Zhou.
\newblock Large ranking games with nonobservable diffusion control.
\newblock {\em Preprint hal-03138716f}, 2021.

\bibitem{Bardi.12}
M.~Bardi.
\newblock Explicit solutions of some linear-quadratic mean field games.
\newblock {\em Netw. Heterog. Media}, 7(2):243--261, 2012.

\bibitem{BayCviZhang.19}
E.~Bayraktar, J.~Cvitani\'{c}, and Y.~Zhang.
\newblock Large tournament games.
\newblock {\em Ann. Appl. Probab.}, 29(6):3695--3744, 2019.

\bibitem{BensoussanFrehseYam.13}
A.~Bensoussan, J.~Frehse, and S.~C.~P. Yam.
\newblock {\em Mean field games and mean field type control theory}.
\newblock Springer Briefs in Mathematics. Springer, New York, 2013.

\bibitem{CampiFischer.18}
L.~Campi and M.~Fischer.
\newblock {$N$}-player games and mean-field games with absorption.
\newblock {\em Ann. Appl. Probab.}, 28(4):2188--2242, 2018.

\bibitem{CardaliaguetDelarueLasryLions.15}
P.~Cardaliaguet, F.~Delarue, J.~M. Lasry, and P.~L. Lions.
\newblock {\em The master equation and the convergence problem in mean field
  games}.
\newblock Annals of Mathematics Studies 381. Princeton University Press, 2019.

\bibitem{CarmonaDelarue.13}
R.~Carmona and F.~Delarue.
\newblock Probabilistic analysis of mean-field games.
\newblock {\em SIAM J. Control Optim.}, 51(4):2705--2734, 2013.

\bibitem{CarmonaDelaRue.17a}
R.~Carmona and F.~Delarue.
\newblock {\em Probabilistic Theory of Mean Field Games with Applications I}.
\newblock Springer, 2017.

\bibitem{CarmonaDelaRue.17b}
R.~Carmona and F.~Delarue.
\newblock {\em Probabilistic Theory of Mean Field Games with Applications II}.
\newblock Springer, 2017.

\bibitem{CarmonaDelarueLacker.17}
R.~Carmona, F.~Delarue, and D.~Lacker.
\newblock Mean field games of timing and models for bank runs.
\newblock {\em Appl. Math. Optim.}, 76(1):217--260, 2017.

\bibitem{CarmonaLacker.15}
R.~Carmona and D.~Lacker.
\newblock A probabilistic weak formulation of mean field games and
  applications.
\newblock {\em Ann. Appl. Probab.}, 25(3):1189--1231, 2015.

\bibitem{CecchinFischer.18}
A.~Cecchin and M.~Fischer.
\newblock Probabilistic approach to finite state mean field games.
\newblock {\em Appl. Math. Optim.}, 81(2):253--300, 2020.

\bibitem{CecchinDaiPraFischerPelino.18}
A.~Cecchin, P.~Dai Pra, M.~Fischer, and G.~Pelino.
\newblock On the convergence problem in mean field games: a two state model
  without uniqueness.
\newblock {\em SIAM J. Control Optim.}, 57(4):2443--2466, 2019.

\bibitem{DelarueFoguenTchuendom.18}
F.~Delarue and R.~Foguen Tchuendom.
\newblock Selection of equilibria in a linear quadratic mean-field game.
\newblock {\em Stochastic Process. Appl.}, 130(2):1000--1040, 2020.

\bibitem{Durrett.10}
R.~Durrett.
\newblock {\em Probability: theory and examples}, volume~31 of {\em Cambridge
  Series in Statistical and Probabilistic Mathematics}.
\newblock Cambridge University Press, Cambridge, fourth edition, 2010.

\bibitem{ElieMastroliaPossamai.19}
R.~Elie, T.~Mastrolia, and D.~Possama\"{\i}.
\newblock A tale of a principal and many, many agents.
\newblock {\em Math. Oper. Res.}, 44(2):440--467, 2019.

\bibitem{FangNoe.16}
D.~Fang and T.~Noe.
\newblock Skewing the odds: Taking risks for rank-based rewards.
\newblock {\em Preprint SSRN:2747496}, 2016.

\bibitem{FangNoeStrack.20}
D.~Fang, T.~Noe, and P.~Strack.
\newblock Turning up the heat: The discouraging effect of competition in
  contests.
\newblock {\em J. Political Econ.}, 128(5):1940--1975, 2020.

\bibitem{FengHobson.15}
H.~Feng and D.~Hobson.
\newblock Gambling in contests modelled with diffusions.
\newblock {\em Decis. Econ. Finance}, 38(1):21--37, 2015.

\bibitem{FengHobson.16a}
H.~Feng and D.~Hobson.
\newblock Gambling in contests with random initial law.
\newblock {\em Ann. Appl. Probab.}, 26(1):186--215, 2016.

\bibitem{FengHobson.16b}
H.~Feng and D.~Hobson.
\newblock Gambling in contests with regret.
\newblock {\em Math. Finance}, 26(3):674--695, 2016.

\bibitem{Fischer.14}
M.~Fischer.
\newblock On the connection between symmetric {$N$}-player games and mean field
  games.
\newblock {\em Ann. Appl. Probab.}, 27(2):757--810, 2017.

\bibitem{GreenStokey.83}
J.~R. Green and N.~L. Stokey.
\newblock A comparison of tournaments and contracts.
\newblock {\em J Polit Econ.}, 91(3):349--364, 1983.

\bibitem{Hall.69}
W.~J. Hall.
\newblock Embedding submartingales in {W}iener processes with drift, with
  applications to sequential analysis.
\newblock {\em J. Appl. Probability}, 6:612--632, 1969.

\bibitem{HuangMalhameCaines.06}
M.~Huang, R.~P. Malham{\'e}, and P.~E. Caines.
\newblock Large population stochastic dynamic games: closed-loop
  {M}c{K}ean-{V}lasov systems and the {N}ash certainty equivalence principle.
\newblock {\em Commun. Inf. Syst.}, 6(3):221--251, 2006.

\bibitem{Lacker.14}
D.~Lacker.
\newblock A general characterization of the mean field limit for stochastic
  differential games.
\newblock {\em Probab. Theory Related Fields}, 165(3-4):581--648, 2016.

\bibitem{Lacker.18b}
D.~Lacker.
\newblock On the convergence of closed-loop {N}ash equilibria to the mean field
  game limit.
\newblock {\em Ann. Appl. Probab.}, 30(4):1693--1761, 2020.

\bibitem{LasryLions.06a}
J.-M. Lasry and P.-L. Lions.
\newblock Jeux \`a champ moyen. {I}. {L}e cas stationnaire.
\newblock {\em C. R. Math. Acad. Sci. Paris}, 343(9):619--625, 2006.

\bibitem{SanMartinNutzTan.18}
J.~San Martin, M.~Nutz, and X.~Tan.
\newblock Convergence to the mean field game limit: A case study.
\newblock {\em Ann. Appl. Probab.}, 30(1):259--286, 2020.

\bibitem{Nutz.16}
M.~Nutz.
\newblock A mean field game of optimal stopping.
\newblock {\em SIAM J. Control Optim.}, 56(2):1206--1221, 2018.

\bibitem{NutzZhang.19}
M.~Nutz and Y.~Zhang.
\newblock A mean field competition.
\newblock {\em Math. Oper. Res.}, 44(4):1245--1263, 2019.

\bibitem{NutzZhang.21a}
M.~Nutz and Y.~Zhang.
\newblock Reward design in risk-taking contests.
\newblock {\em Preprint arXiv:2102.03417v1}, 2021.

\bibitem{Obloj.04}
J.~Ob{\l}{\'o}j.
\newblock The {S}korokhod embedding problem and its offspring.
\newblock {\em Probab. Surv.}, 1:321--390, 2004.

\bibitem{Resnick.07}
S.~I. Resnick.
\newblock {\em Heavy-tail phenomena}.
\newblock Springer Series in Operations Research and Financial Engineering.
  Springer, New York, 2007.

\bibitem{Seel.15}
C.~Seel.
\newblock Gambling in contests with heterogeneous loss constraints.
\newblock {\em Economics Letters}, 136:154 -- 157, 2015.

\bibitem{SeelStrack.13}
C.~Seel and P.~Strack.
\newblock Gambling in contests.
\newblock {\em J. Econ. Theory}, 148(5):2033--2048, 2013.

\bibitem{Sun.06}
Y.~Sun.
\newblock The exact law of large numbers via {F}ubini extension and
  characterization of insurable risks.
\newblock {\em J. Econom. Theory}, 126(1):31--69, 2006.

\bibitem{Vojnovic.2016}
M.~Vojnovi\'{c}.
\newblock {\em Contest Theory: Incentive Mechanisms and Ranking Methods}.
\newblock Cambridge University Press, 2016.

\end{thebibliography}
\bibliographystyle{plain}
\end{document}